\documentclass[a4paper,10pt]{amsart}

\usepackage[utf8]{inputenc}
\usepackage{url}

\usepackage[expansion=false]{microtype}
\usepackage{amssymb,amsmath,amsopn,amsxtra,amsthm,amsfonts}
\usepackage{xr}
\usepackage[mathcal]{euscript}
\usepackage{mathalfa}

\usepackage[english, status=final, nomargin, inline]{fixme}
\fxusetheme{color}

\usepackage[pdfusetitle,unicode,hidelinks]{hyperref}

\newcommand{\yo}{\text{\usefont{U}{min}{m}{n}\symbol{'210}}}
\DeclareFontFamily{U}{min}{}
\DeclareFontShape{U}{min}{m}{n}{<-> udmj30}{}

\usepackage{enumitem}
\setlist[enumerate,1]{label={(\arabic*)},itemsep=\parskip} 
\setlist[itemize,1]{itemsep=\parskip} 
\newlist{thmlist}{enumerate}{2}
\setlist[thmlist,1]{label={\em(\roman*)},ref={(\roman*)},%
  itemsep=\parskip,leftmargin=*,align=left}
\setlist[thmlist,2]{label={\em(\alph*)},ref={(\alph*)},%
  itemsep=\parskip,leftmargin=*,align=left,topsep=0.1cm}
\newlist{remlist}{enumerate}{2}
\setlist[remlist,1]{label={(\roman*)},ref={(\roman*)},itemsep=\parskip,%
  leftmargin=*,align=left}
\setlist[remlist,2]{label={(\alph*)},ref={(\alph*)},itemsep=\parskip,%
  leftmargin=*,align=left,topsep=0.1cm}

\makeatletter
\let\c@equation\c@subsubsection

\makeatother

\newtheorem{cor}[subsubsection]{Corollary}
\newtheorem{lem}[subsubsection]{Lemma}
\newtheorem{prop}[subsubsection]{Proposition}

\newtheorem{conj}[subsubsection]{Conjecture}
\newtheorem{thm}[subsubsection]{Theorem}

\newtheorem*{claim*}{Claim}

\theoremstyle{definition}
\newtheorem{defn}[subsubsection]{Definition}
\newtheorem{quest}[subsubsection]{Question}

\newtheorem{rem}[subsubsection]{Remark}
\newtheorem{constr}[subsubsection]{Construction}
\newtheorem{exam}[subsubsection]{Example}

\renewcommand{\eqref}[1]{(\ref{#1})}

\usepackage{tikz}
\usetikzlibrary{matrix}
\usepackage{tikz-cd}

\usepackage{xspace}

\usepackage{xifthen}

\usepackage{xparse}

\newcommand{\nc}{\newcommand}
\nc{\renc}{\renewcommand}
\nc{\ssec}{\subsection}
\nc{\sssec}{\subsubsection}
\nc{\on}{\operatorname}
\nc{\term}[1]{#1\xspace}

\DeclareMathSymbol{A}{\mathalpha}{operators}{`A}
\DeclareMathSymbol{B}{\mathalpha}{operators}{`B}
\DeclareMathSymbol{C}{\mathalpha}{operators}{`C}
\DeclareMathSymbol{D}{\mathalpha}{operators}{`D}
\DeclareMathSymbol{E}{\mathalpha}{operators}{`E}
\DeclareMathSymbol{F}{\mathalpha}{operators}{`F}
\DeclareMathSymbol{G}{\mathalpha}{operators}{`G}
\DeclareMathSymbol{H}{\mathalpha}{operators}{`H}
\DeclareMathSymbol{I}{\mathalpha}{operators}{`I}
\DeclareMathSymbol{J}{\mathalpha}{operators}{`J}
\DeclareMathSymbol{K}{\mathalpha}{operators}{`K}
\DeclareMathSymbol{L}{\mathalpha}{operators}{`L}
\DeclareMathSymbol{M}{\mathalpha}{operators}{`M}
\DeclareMathSymbol{N}{\mathalpha}{operators}{`N}
\DeclareMathSymbol{O}{\mathalpha}{operators}{`O}
\DeclareMathSymbol{P}{\mathalpha}{operators}{`P}
\DeclareMathSymbol{Q}{\mathalpha}{operators}{`Q}
\DeclareMathSymbol{R}{\mathalpha}{operators}{`R}
\DeclareMathSymbol{S}{\mathalpha}{operators}{`S}
\DeclareMathSymbol{T}{\mathalpha}{operators}{`T}
\DeclareMathSymbol{U}{\mathalpha}{operators}{`U}
\DeclareMathSymbol{V}{\mathalpha}{operators}{`V}
\DeclareMathSymbol{W}{\mathalpha}{operators}{`W}
\DeclareMathSymbol{X}{\mathalpha}{operators}{`X}
\DeclareMathSymbol{Y}{\mathalpha}{operators}{`Y}
\DeclareMathSymbol{Z}{\mathalpha}{operators}{`Z}

\nc{\sA}{\ensuremath{\mathcal{A}}\xspace}
\nc{\sB}{\ensuremath{\mathcal{B}}\xspace}
\nc{\sC}{\ensuremath{\mathcal{C}}\xspace}
\nc{\sD}{\ensuremath{\mathcal{D}}\xspace}
\nc{\sE}{\ensuremath{\mathcal{E}}\xspace}
\nc{\sF}{\ensuremath{\mathcal{F}}\xspace}
\nc{\sG}{\ensuremath{\mathcal{G}}\xspace}
\nc{\sH}{\ensuremath{\mathcal{H}}\xspace}
\nc{\sI}{\ensuremath{\mathcal{I}}\xspace}
\nc{\sJ}{\ensuremath{\mathcal{J}}\xspace}
\nc{\sK}{\ensuremath{\mathcal{K}}\xspace}
\nc{\sL}{\ensuremath{\mathcal{L}}\xspace}
\nc{\sM}{\ensuremath{\mathcal{M}}\xspace}
\nc{\sN}{\ensuremath{\mathcal{N}}\xspace}
\nc{\sO}{\ensuremath{\mathcal{O}}\xspace}
\nc{\sP}{\ensuremath{\mathcal{P}}\xspace}
\nc{\sQ}{\ensuremath{\mathcal{Q}}\xspace}
\nc{\sR}{\ensuremath{\mathcal{R}}\xspace}
\nc{\sS}{\ensuremath{\mathcal{S}}\xspace}
\nc{\sT}{\ensuremath{\mathcal{T}}\xspace}
\nc{\sU}{\ensuremath{\mathcal{U}}\xspace}
\nc{\sV}{\ensuremath{\mathcal{V}}\xspace}
\nc{\sW}{\ensuremath{\mathcal{W}}\xspace}
\nc{\sX}{\ensuremath{\mathcal{X}}\xspace}
\nc{\sY}{\ensuremath{\mathcal{Y}}\xspace}
\nc{\sZ}{\ensuremath{\mathcal{Z}}\xspace}

\nc{\bA}{\ensuremath{\mathbf{A}}\xspace}
\nc{\bB}{\ensuremath{\mathbf{B}}\xspace}
\nc{\bC}{\ensuremath{\mathbf{C}}\xspace}
\nc{\bD}{\ensuremath{\mathbf{D}}\xspace}
\nc{\bE}{\ensuremath{\mathbf{E}}\xspace}
\nc{\bF}{\ensuremath{\mathbf{F}}\xspace}
\nc{\bG}{\ensuremath{\mathbf{G}}\xspace}
\nc{\bH}{\ensuremath{\mathbf{H}}\xspace}
\nc{\bI}{\ensuremath{\mathbf{I}}\xspace}
\nc{\bJ}{\ensuremath{\mathbf{J}}\xspace}
\nc{\bK}{\ensuremath{\mathbf{K}}\xspace}
\nc{\bL}{\ensuremath{\mathbf{L}}\xspace}
\nc{\bM}{\ensuremath{\mathbf{M}}\xspace}
\nc{\bN}{\ensuremath{\mathbf{N}}\xspace}
\nc{\bO}{\ensuremath{\mathbf{O}}\xspace}
\nc{\bP}{\ensuremath{\mathbf{P}}\xspace}
\nc{\bQ}{\ensuremath{\mathbf{Q}}\xspace}
\nc{\bR}{\ensuremath{\mathbf{R}}\xspace}
\nc{\bS}{\ensuremath{\mathbf{S}}\xspace}
\nc{\bT}{\ensuremath{\mathbf{T}}\xspace}
\nc{\bU}{\ensuremath{\mathbf{U}}\xspace}
\nc{\bV}{\ensuremath{\mathbf{V}}\xspace}
\nc{\bW}{\ensuremath{\mathbf{W}}\xspace}
\nc{\bX}{\ensuremath{\mathbf{X}}\xspace}
\nc{\bY}{\ensuremath{\mathbf{Y}}\xspace}
\nc{\bZ}{\ensuremath{\mathbf{Z}}\xspace}

\nc{\dA}{\ensuremath{\mathds{A}}\xspace}
\nc{\dB}{\ensuremath{\mathds{B}}\xspace}
\nc{\dC}{\ensuremath{\mathds{C}}\xspace}
\nc{\dD}{\ensuremath{\mathds{D}}\xspace}
\nc{\dE}{\ensuremath{\mathds{E}}\xspace}
\nc{\dF}{\ensuremath{\mathds{F}}\xspace}
\nc{\dG}{\ensuremath{\mathds{G}}\xspace}
\nc{\dH}{\ensuremath{\mathds{H}}\xspace}
\nc{\dI}{\ensuremath{\mathds{I}}\xspace}
\nc{\dJ}{\ensuremath{\mathds{J}}\xspace}
\nc{\dK}{\ensuremath{\mathds{K}}\xspace}
\nc{\dL}{\ensuremath{\mathds{L}}\xspace}
\nc{\dM}{\ensuremath{\mathds{M}}\xspace}
\nc{\dN}{\ensuremath{\mathds{N}}\xspace}
\nc{\dO}{\ensuremath{\mathds{O}}\xspace}
\nc{\dP}{\ensuremath{\mathds{P}}\xspace}
\nc{\dQ}{\ensuremath{\mathds{Q}}\xspace}
\nc{\dR}{\ensuremath{\mathds{R}}\xspace}
\nc{\dS}{\ensuremath{\mathds{S}}\xspace}
\nc{\dT}{\ensuremath{\mathds{T}}\xspace}
\nc{\dU}{\ensuremath{\mathds{U}}\xspace}
\nc{\dV}{\ensuremath{\mathds{V}}\xspace}
\nc{\dW}{\ensuremath{\mathds{W}}\xspace}
\nc{\dX}{\ensuremath{\mathds{X}}\xspace}
\nc{\dY}{\ensuremath{\mathds{Y}}\xspace}
\nc{\dZ}{\ensuremath{\mathds{Z}}\xspace}

\nc{\bbA}{\ensuremath{\mathbb{A}}\xspace}
\nc{\bbB}{\ensuremath{\mathbb{B}}\xspace}
\nc{\bbC}{\ensuremath{\mathbb{C}}\xspace}
\nc{\bbD}{\ensuremath{\mathbb{D}}\xspace}
\nc{\bbE}{\ensuremath{\mathbb{E}}\xspace}
\nc{\bbF}{\ensuremath{\mathbb{F}}\xspace}
\nc{\bbG}{\ensuremath{\mathbb{G}}\xspace}
\nc{\bbH}{\ensuremath{\mathbb{H}}\xspace}
\nc{\bbI}{\ensuremath{\mathbb{I}}\xspace}
\nc{\bbJ}{\ensuremath{\mathbb{J}}\xspace}
\nc{\bbK}{\ensuremath{\mathbb{K}}\xspace}
\nc{\bbL}{\ensuremath{\mathbb{L}}\xspace}
\nc{\bbM}{\ensuremath{\mathbb{M}}\xspace}
\nc{\bbN}{\ensuremath{\mathbb{N}}\xspace}
\nc{\bbO}{\ensuremath{\mathbb{O}}\xspace}
\nc{\bbP}{\ensuremath{\mathbb{P}}\xspace}
\nc{\bbQ}{\ensuremath{\mathbb{Q}}\xspace}
\nc{\bbR}{\ensuremath{\mathbb{R}}\xspace}
\nc{\bbS}{\ensuremath{\mathbb{S}}\xspace}
\nc{\bbT}{\ensuremath{\mathbb{T}}\xspace}
\nc{\bbU}{\ensuremath{\mathbb{U}}\xspace}
\nc{\bbV}{\ensuremath{\mathbb{V}}\xspace}
\nc{\bbW}{\ensuremath{\mathbb{W}}\xspace}
\nc{\bbX}{\ensuremath{\mathbb{X}}\xspace}
\nc{\bbY}{\ensuremath{\mathbb{Y}}\xspace}
\nc{\bbZ}{\ensuremath{\mathbb{Z}}\xspace}

\nc{\mrm}[1]{\ensuremath{\mathrm{#1}}\xspace}
\nc{\mbf}[1]{\ensuremath{\mathbf{#1}}\xspace}
\nc{\mcal}[1]{\ensuremath{\mathcal{#1}}\xspace}
\nc{\msc}[1]{\ensuremath{\mathscr{#1}}\xspace}

\renc{\bar}[1]{\overline{#1}}

\nc{\sub}{\subset}
\nc{\too}{\longrightarrow}
\nc{\hook}{\hookrightarrow}
\nc*{\hooklongrightarrow}{\ensuremath{\lhook\joinrel\relbar\joinrel\rightarrow}}
\nc{\hooklong}{\hooklongrightarrow}
\nc{\twoheadlongrightarrow}{\relbar\joinrel\twoheadrightarrow}
\nc{\shiso}{\approx}
\nc{\isoto}{\xrightarrow{\sim}}
\nc{\isofrom}{\xleftarrow{\sim}}
\renc{\ge}{\geqslant}
\renc{\le}{\leqslant}
\renc{\geq}{\geqslant}
\renc{\leq}{\leqslant}

\nc{\id}{\mathrm{id}}
\DeclareMathOperator{\Ker}{\mathrm{Ker}}
\DeclareMathOperator{\Coker}{\mathrm{Coker}}

\DeclareMathOperator{\Hom}{\mathrm{Hom}}
\nc{\uHom}{\underline{\smash{\Hom}}}
\DeclareMathOperator{\Maps}{\mathrm{Maps}}

\DeclareMathOperator{\End}{\mathrm{End}}

\nc{\Pre}{\mathrm{PSh}{}}
\nc{\Shv}{\mathrm{Shv}{}}
\nc{\uEnd}{\underline{\smash{\End}}}

\renc{\lim}{\operatorname*{lim}}
\nc{\colim}{\operatorname*{colim}}
\nc{\Cofib}{\on{Cofib}}
\nc{\Fib}{\on{Fib}}
\nc{\initial}{\varnothing}
\nc{\op}{\mathrm{op}}

\renc{\coprod}{\sqcup}

\nc{\bDelta}{\mbf{\Delta}}
\nc{\DM}{\mbf{DM}}
\nc{\eff}{\mathrm{eff}}
\nc{\veff}{\mathrm{veff}}
\nc{\cyc}{{\mrm{cyc}}}
\nc{\corr}{{\on{corr}}}
\nc{\ft}{\mrm{ft}}
\nc{\flf}{\mrm{flf}}
\nc{\fet}{{\mrm{f\acute et}}}
\nc{\fsyn}{{\mrm{fsyn}}}
\nc{\syn}{{\mrm{syn}}}
\nc{\lci}{{\mrm{lci}}}
\nc{\Perf}{\mbf{Perf}}
\nc{\perf}{\mrm{perf}}
\nc{\oblv}{\mrm{oblv}}
\nc{\exact}{\on{exact}}

\nc{\F}{{\on{F}}}
\nc{\clopen}{{\mrm{clopen}}}
\nc{\B}{\mrm{B}}
\nc{\D}{\mrm{D}}
\nc{\Fin}{\on{Fin}}
\nc{\fin}{\mrm{fin}}
\nc{\Cut}{\on{Cut}}
\nc{\Cart}{\on{Cart}}
\nc{\pairs}{\mathsf{pairs}}
\nc{\Pairs}{\mathrm{Pair}}
\nc{\Trip}{\mathrm{Trip}}
\nc{\Lab}{\mathrm{Lab}}
\nc{\SL}{\mathrm{SL}}
\nc{\coCart}{\mathrm{coCart}}
\nc{\RKE}{\mathrm{RKE}}
\nc{\strict}{\mathrm{strict}}
\nc{\Emb}{\mathrm{Emb}}
\nc{\Split}{\mathrm{Split}}
\nc{\Set}{\mathrm{Set}}
\nc{\sSets}{\mathrm{sSets}}
\nc{\pb}{\mathrm{pb}}
\nc{\fib}{\mathrm{fib}}
\nc{\diff}{\mrm{diff}}
\nc{\gp}{\mrm{gp}}
\nc{\map}{\mrm{map}}

\nc{\mgp}{\mrm{mot-gp}}
\nc{\FSyn}{\mrm{FSyn}}
\nc{\FEt}{\mrm{FEt}}
\nc{\Spc}{\mrm{Spc}}
\nc{\Ob}{\mrm{Ob}}
\nc{\Spt}{\mrm{Spt}}
\nc{\T}{\bT}
\nc{\suspinf}{\Sigma^\infty}
\nc{\h}{\mrm{h}}
\nc{\uhom}{\underline{\mathrm{Hom}}}
\nc{\umap}{\underline{\mathrm{Maps}}}
\renc{\H}{\bH}
\nc{\Einfty}{{\sE_\infty}}
\nc{\Eone}{{\sE_1}}
\nc{\Stab}{\mrm{Stab}}
\nc{\lax}{{\mrm{lax}}}
\nc{\cocart}{{\mrm{cocart}}}
\nc{\Sch}{\on{Sch}}
\nc{\Fr}{\on{Fr}}
\nc{\A}{\mathbf{A}}
\nc{\N}{\mathbf{N}}
\nc{\Z}{\mathbf{Z}}
\nc{\Q}{\mathbf{Q}}
\nc{\Oo}{\mathcal{O}} 
\nc{\Fscr}{\mathcal{F}}
\nc{\Gscr}{\mathcal{G}}
\nc{\Ll}{\mathcal{L}} 
\nc{\Mm}{\mathcal{M}} 
\nc{\mm}{\mathrm{m}} 
\nc{\K}{\mrm{K}} 
\nc{\W}{\mrm{W}} 
\nc{\red}{{\on{red}}}
\nc{\Voev}{{\on{Voev}}}
\nc{\Corr}{\mrm{Corr}}
\nc{\Span}{\mathbf{Corr}}
\nc{\Gap}{\mrm{Gap}}
\nc{\Corrfr}{\Corr^{\fr}}
\nc{\Corrvfr}{\Corr^{\Vfr}}
\nc{\Spec}{\on{Spec}}
\nc{\Sm}{\on{Sm}}
\nc{\Gm}{\mathbf{G}_{\on{m}}}
\renc{\P}{\bP}
\nc{\nis}{\mathrm{nis}}
\nc{\zar}{\mathrm{zar}}
\nc{\et}{\mathrm{\acute et}}
\nc{\all}{\mathrm{all}}
\nc{\fold}{\mathrm{fold}}
\nc{\Fun}{\mathrm{Fun}}
\nc{\Nat}{\mathrm{Nat}}
\nc{\Ho}{\mathrm{Ho}}
\nc{\Segal}{\mathrm{Segal}}
\nc{\Mon}{\mrm{Mon}{}}
\nc{\Ab}{\mrm{Ab}}
\nc{\Sh}{\on{Sh}}
\nc{\M}{\mrm{M}}
\nc{\Lhtp}{L_{\A^1}}
\nc{\Lmot}{L_{\mrm{mot}}}
\nc{\mot}{\mrm{mot}}
\nc{\SH}{\mbf{SH}}
\nc{\RR}{\mbf{R}}
\nc{\CC}{\mbf{C}}
\nc{\Mod}{\mbf{Mod}}
\nc{\QCoh}{\mbf{QCoh}}
\nc{\MonUnit}{\mbf{1}}
\nc{\1}{\mbf{1}}
\nc{\tr}{\on{tr}}
\nc{\cotr}{\mrm{cotr}}
\nc{\vop}{\mrm{vop}}
\nc{\fr}{{\on{fr}}}
\nc{\Ar}{\mrm{Ar}}
\nc{\Vfr}{\on{Vfr}}
\nc{\frdiff}{{\on{frdiff}}}
\nc{\frGys}{\on{frGys}}
\nc{\SHfr}{\SH^{\fr}}
\nc{\SHfrdiff}{\SH^{\frdiff}}
\nc{\SHfrGys}{\SH^{\frGys}}
\nc{\InftyCat}{(\mathrm{\infty,1)\textnormal{-}Cat}}
\nc{\TriCat}{\mathrm{TriCat}}
\nc{\oneCat}{\mathrm{1\textnormal{-}Cat}}
\nc{\Cat}{\mathrm{Cat}}
\nc{\Th}{\on{Th}}

\nc{\CMon}{\mrm{CMon}{}}
\nc{\CAlg}{\mrm{CAlg}{}}
\nc{\MGL}{\mrm{MGL}}
\nc{\Seg}{\mrm{Seg}{}}
\nc{\GW}{\mrm{GW}{}}
\nc{\Tw}{\mrm{Tw}}
\nc{\sslash}{/\mkern-6mu/}
\nc{\PrL}{\mrm{Pr}^\mrm{L}}
\nc{\PrR}{\mrm{Pr}^\mrm{R}}
\nc{\pr}{\mrm{pr}}
\let\phi\varphi

\nc\efr{\mrm{efr}}
\nc\nfr{\mrm{nfr}}
\nc\dfr{\mrm{fr}}
\nc\tfr{\mrm{tfr}}
\nc\Vect{\mrm{Vect}}
\nc\sVect{\mrm{sVect}}
\nc{\fix}{\mrm{fix}}
\nc{\ho}{\mrm{h}}
\nc\Mfd{\mrm{Mfd}}
\nc{\PSh}{\mrm{PSh}}
\nc{\hzmw}{H \tilde\Z{}}
\nc{\Cor}{\mrm{Cor}{}}
\nc{\cormw}{\mrm{\widetilde{Cor}}{}}
\nc{\Chw}{\mrm{\widetilde{CH}}{}}
\nc{\Ex}{\mrm{Ex}}
\nc{\BM}{\mrm{BM}}

\nc{\Pic}{\mrm{Pic}}
\nc{\Br}{\mrm{Br}}
\nc{\pur}{\mathfrak p}
\nc{\angles}[1]{\langle #1\rangle}
\nc{\inv}[1]{[\tfrac{1}{#1}]}
\nc{\pinv}{\inv{p}}
\nc{\cinv}{\inv{p}}
\nc{\Sph}{\on{Sph}}
\nc{\KGL}{\mrm{KGL}}
\nc{\KH}{\mrm{KH}}
\nc{\Flag}{\mrm{Flag}}
\nc{\Pro}{\mrm{Pro}}
\nc{\Frac}{\mrm{Frac}}

\nc{\arc}{\mrm{arc}}
\nc{\rarc}{\mrm{rarc}}
\nc{\cdarc}{\mrm{cdarc}}
\nc{\vv}{\mrm{v}}
\nc{\rv}{\mrm{rv}}
\nc{\cdv}{\mrm{cdv}}
\nc{\hh}{\mrm{h}}
\nc{\cdh}{\mrm{cdh}}
\nc{\rh}{\mathrm{rh}}
\nc{\Et}{\mathrm{Et}}
\nc{\Nis}{\mathrm{Nis}}
\nc{\Zar}{\mathrm{Zar}}
\nc{\cdp}{\mathrm{cdp}}

\nc{\RZ}{\mathrm{RZ}}
\nc{\qcqs}{\mathrm{qcqs}}
\nc{\aff}{\mathrm{aff}}
\nc{\cl}{\mathrm{cl}}
\nc{\Val}{\mathrm{Val}}
\nc{\GFin}{\mathrm{GFin}{}}
\nc{\Proj}{\mathrm{Proj}}

\nc{\inftyCat}{\term{$\infty$-category}}
\nc{\inftyCats}{\term{$\infty$-categories}}

\nc{\inftyOneCat}{\term{$(\infty,1)$-category}}
\nc{\inftyOneCats}{\term{$(\infty,1)$-categories}}

\nc{\inftyGrpd}{\term{$\infty$-groupoid}}
\nc{\inftyGrpds}{\term{$\infty$-groupoids}}

\nc{\inftyTop}{\term{$\infty$-topos}}
\nc{\inftyTops}{\term{$\infty$-toposes}}

\nc{\inftyTwoCat}{\term{$(\infty,2)$-category}}
\nc{\inftyTwoCats}{\term{$(\infty,2)$-categories}}

\title{On nilpotent extensions of $\infty$-categories and the cyclotomic trace}

\author[E. Elmanto]{Elden Elmanto}
\address{Department of Mathematics\\
Harvard University\\
1 Oxford St.\
Cambridge, MA 02138\\
USA}
\email{\href{mailto:elmanto@math.harvard.edu}{elmanto@math.harvard.edu}}
\urladdr{\url{https://www.eldenelmanto.com/}}

\author[V. Sosnilo]{Vladimir Sosnilo}
\address{Laboratory ``Modern Algebra and Applications''\\
St. Petersburg State University\\
14th line, 29B\\
199178 Saint Petersburg\\
Russia}
\email{\href{mailto:vsosnilo@gmail.com}{vsosnilo@gmail.com}}

\begin{document}

\bibliographystyle{alphamod}

\begin{abstract} We do three things in this paper: (1)  study the analog of localization sequences (in the sense of algebraic $K$-theory of stable $\infty$-categories) for additive $\infty$-categories, (2) define the notion of nilpotent extensions for suitable $\infty$-categories and furnish interesting examples such as categorical square-zero extensions, and (3) use (1) and (2) to extend the Dundas-Goodwillie-McCarthy theorem for stable $\infty$-categories which are not monogenically generated (such as the stable $\infty$-category of Voevodsky's motives or the stable $\infty$-category of perfect complexes on some algebraic stacks). The key input in our paper is Bondarko's notion of weight structures which provides a ``ring-with-many-objects" analog of a connective $\mathbb{E}_1$-ring spectrum. As applications, we prove cdh descent results for truncating invariants of stacks extending the work of Hoyois-Krishna for homotopy $K$-theory, and establish new cases of Blanc's lattice conjecture.
\end{abstract}

\maketitle

\tableofcontents

\section{Introduction}

Every application of trace methods in algebraic $K$-theory goes through the celebrated Dundas-Goodwillie-McCarthy (DGM) theorem \cite{goodwillie, dundas, mccarthy, dgm-local}:

\begin{thm}[Dundas-Goodwillie-McCarthy]\label{thm:dgm} Let $A \rightarrow B$ be a morphism of connective $\mathbb{E}_1$-ring spectra such that the map $\pi_0(A) \rightarrow \pi_0(B)$ is surjective and the kernel is a nilpotent ideal of $\pi_0(A)$. Then the trace map from algebraic $K$-theory to topological cyclic homology ($TC$)
\[
\mrm{tr}:K \rightarrow TC
\]
induces a cartesian square
\[
\begin{tikzcd}
K(A) \ar{d} \ar{r} & TC(A) \ar{d}\\
K(B) \ar{r} & TC(B).
\end{tikzcd}
\]
\end{thm}

The goal of this paper is to extend Theorem~\ref{thm:dgm} to a more general setting. To motivate our results, let us recall that $TC$ and $K$-theory are both instances of {\bf localizing invariants}: functors from the $\infty$-category of small stable idempotent complete $\infty$-categories and exact functors to the $\infty$-category of spectra:
\[
\on{E}: \mbf{Cat}_{\infty}^{\perf} \rightarrow \Spt,
\]
which convert exact sequences\footnote{An exact sequence in $\mbf{Cat}_{\infty}^{\perf}$ is a sequence $\sA \rightarrow \sB \rightarrow \sC$ whose composite is zero and the induced functor $\sB/\sA \rightarrow \sC$ is an equivalence after idempotent completion.}  in $\mbf{Cat}_{\infty}^{\perf}$ to cofiber sequences in spectra. In particular, both invariants are defined, more generally, on $\mbf{Cat}_{\infty}^{\perf}$ rather than just on ring spectra. We are motivated by the following general question:

\begin{quest}\label{quest:dgm} Fix $f: \sA \rightarrow \sB,$ an exact functor of stable $\infty$-categories. Is the induced square:
\[
\begin{tikzcd}
K(\sA) \ar{d} \ar{r} & TC(\sA) \ar{d}\\
K(\sB) \ar{r} & TC(\sB).
\end{tikzcd}
\]
cartesian? Equivalently, setting
\[
K^{\mrm{inv}}:= \mrm{Fib}(K \xrightarrow{\mrm{tr}} TC),
\]
is $f$ converted to an equivalence after taking $K^{\mrm{inv}}$?
\end{quest}

Theorem~\ref{thm:dgm} provides a large collection of functors for which the answer to Question~\ref{quest:dgm} is positive. Namely, functors of the form
\[
-\otimes_A B:\Perf_A \rightarrow \Perf_B,
\]
induced by a map $A \rightarrow B$ satisfying the hypotheses of Theorem~\ref{thm:dgm}. Here $\Perf_A \subset \mbf{RMod}_A$ is the subcategory of perfect objects in the $\infty$-category of right $A$-modules.

Since we only know positive answers to Question~\ref{quest:dgm} when the domain and codomain of $f$ are of the above form, some constraints are placed on these categories. Indeed, if $\sC \in \mbf{Cat}_{\infty}^{\perf}$, we pass to the $\infty$-category of ind-objects to obtain a presentable, stable $\infty$-category $\mrm{Ind}(\sC)$. In this situation, an important theorem due to Schwede and Shipley \cite{SchwedeShipley2003} (we use its formulation from \cite[Theorem 7.1.2.1]{HA}) asserts that $\mrm{Ind}(\sC)$ is of the form $\mbf{RMod}_R$ (and thus, taking compact objects, gives an equivalence $\sC \simeq \Perf_R$)
as soon as there exists an object $X \in \sC$ which satisfies the following condition:
\begin{itemize}
\item if $Y \in \mrm{Ind}(\sC)$ is an object such that $\mrm{Ext}_{\sC}^n(X, Y) = 0$ for all $n \in \bbZ$ then $Y \simeq 0$.
\end{itemize}
The $\mathbb{E}_1$-ring spectrum $R$ is then obtained as the endomorphism spectrum of the object $X$. In particular, Theorem~\ref{thm:dgm} is, thus far, only applicable to categories with a single compact generator.
%
%
%
%
%
%
%
On the other hand, there are at least a couple of scenarios of interest in algebraic geometry, where this is not the case.

\begin{enumerate}
\item Let $X$ be a quotient stack $[\Spec R/G]$ where $G$ is a reasonably ``good" group scheme (e.g. embeddable, linearly reductive group) and $R$ a commutative ring with an action of $G$. We are interested in the $\infty$-category of perfect complexes on $X$ which, in the ``good" situations above, coincides with the subcategory of compact objects of the derived $\infty$-category of quasicoherent sheaves:
\[
\Perf_{X} \simeq (\bD_{\mrm{qc}, X})^{\omega}.
\]
In this case, the derived $\infty$-category of quasicoherent sheaves on $X$ admits generators given by
$G$-equivariant bundles in the image of the base change functor
\[
f^*:\bD_{\mrm{qc}, BG} \rightarrow \bD_{\mrm{qc}, X}.
\]
Therefore, $\Perf_{X}$ is generated as a stable subcategory of $\bD_{\mrm{qc}, X}$ closed under retracts, by a \emph{set} of $G$-vector bundles on $\Spec R$ but not necessarily by any single vector bundle. We refer the reader to Section~\ref{ssec:stacks-weights} for a review.
\item Let $R$ be a ring and $k$ be a field. We have Voevodsky's $\infty$-category of motives with coefficients in $R$, denoted by $\DM(\Spec k; R)$, which contains the subcategory of {\bf geometric motives}
\[
\DM_{\mrm{gm}}(\Spec k, R) \subset \DM(\Spec k, R).
\]
The latter is generated as a stable subcategory closed under retracts, by objects
\[
\{ M(X)(q): X\text{ is a smooth $k$-scheme}, q \in \bbZ \}.
\]
The $\infty$-category $\DM_{\mrm{gm}}(\Spec k, R)$ coincides with the compact objects of $ \DM(\Spec k, R)$ and the set of objects above are compact generators of $\DM(\Spec k, R)$. We refer the reader to Section~\ref{sec:motivic} for a review. This is also the case for other ``motivic categories" such as the relative versions of $\DM$, modules over various motivic ring spectra or constructible $\ell$-adic sheaves. 
\end{enumerate}

%
%

In these cases, the $K$-theory and $TC$ are not the $K$-theory and $TC$ of an $\mathbb{E}_1$-ring spectrum and thus trace methods and Theorem~\ref{thm:dgm} do not directly apply. Our primary goal in this paper is to improve this situation. To state our main result, recall the following definition from \cite[Definition 3.1]{Land_2019}:

\begin{defn}\label{def:truncating} Let $\on{E}: \mbf{Cat}_{\infty}^{\perf} \rightarrow \Spt$ be a localizing invariant, and set $\on{E}(R) := \on{E}(\Perf_R)$. Then $E$ is said to be {\bf truncating} if for all connective $\mathbb{E}_1$-ring spectra $R$, the canonical map $R \rightarrow \pi_0(R)$ induces an equivalence:
\[
\on{E}(R) \xrightarrow{\simeq} \on{E}(\pi_0(R)).
\]
\end{defn}

Thanks to Theorem~\ref{thm:dgm}, the fiber of the cyclotomic trace
\[
K^{\mrm{inv}}:= \mrm{Fib}(K \xrightarrow{\mrm{tr}} TC),
\]
is an example of a truncating invariant and is of primary interest in this paper. Our main result produces a large collection of functors in $\mbf{Cat}_{\infty}^{\perf}$ which are taken to equivalences after applying any truncating invariant and thus provides positive answers to Question~\ref{quest:dgm}. The conditions we put on these functors depend on an auxiliary choice of weight structures (in the sense of Bondarko) on both the domain and the target, and these functors are called {\bf nilpotent extensions of boundedly weighted stable $\infty$-categories} (see Definition~\ref{def:nilp-stab} for details). We display these functors as:
\[
f:(\sA, w) \rightarrow (\sB, w').
\]

\begin{thm}[Theorem~\ref{thm:dgm-weights}] \label{thm:main} Let $f:(\sA, w) \rightarrow (\sB, w')$ be a nilpotent extension of boundedly weighted stable $\infty$-categories and $E$ be a truncating invariant. Then the induced map
\[
\on{E}(\sA) \rightarrow \on{E}(\sB),
\]
is an equivalence.
\end{thm}

As already mentioned, $K^{\mrm{inv}}$ is a truncating invariant. Therefore, we obtain the following corollary, which is a generalization of Theorem~\ref{thm:dgm}:

\begin{cor}\label{cor:multdgm} Let $f:(\sA, w) \rightarrow (\sB, w)$ be a nilpotent extension of boundedly weighted stable $\infty$-categories. Then the diagram
\[
\begin{tikzcd}
K(\sA) \ar{d} \ar{r} & TC(\sA) \ar{d}\\
K(\sB) \ar{r} & TC(\sB)
\end{tikzcd}
\]
is a cartesian square.
\end{cor}

Let us now elaborate on what Corollary~\ref{cor:multdgm} has to say about the two scenarios we have
explained above.

%
%

%

\begin{enumerate}
\item If $E$ is a localizing invariant and $R$ is a connective $\bbE_{\infty}$-ring with a $G$-action, we set
\[
\on{E}^G(R) := \on{E}(\Perf_{[\Spec R/G]}),
\]
which is a version of $G$-equivariant $E$-theory. To sharpen our discussions, we fix a base discrete commutative ring $k$ and we let $G$ be a linearly reductive, embeddable group scheme over $k$; these notions regarding groups schemes will be reviewed in Section~\ref{ssec:stacks-weights}. If $R$ is an $\bbE_{\infty}$-$k$-algebra (for example, a commutative, discrete $k$-algebra) with a $G$-action, then the $\infty$-category $\Perf_{[\Spec R/G]}$ is equipped with a canonical weight structure. Any $G$-equivariant\footnote{In the sense that $G$ acts on $\Spec R$ and $\Spec S$ as spectral $k$-schemes and the map $\Spec S \rightarrow \Spec R$ is $G$-equivariant.} map of connective $\bbE_{\infty}$-$k$-algebras $R \rightarrow S$, whose underlying map is a nilpotent extension in the sense that the kernel of
\[
\pi_0(R) \rightarrow \pi_0(S)
\]
is a nilpotent ideal, induces a functor
\[
\Perf_{[\Spec R/G]} \rightarrow \Perf_{[\Spec S/G]},
\]
which is an example of a nilpotent extension of boundedly weighted stable $\infty$-categories. Corollary~\ref{cor:multdgm} then yields:


\begin{cor}[See Example~\ref{exam:nilextension-of-stacks}]\label{cor:stacks} As above, the square  \[
  \begin{tikzcd}
  K^G(R) \ar{d} \ar{r} & TC^G(R) \ar{d}\\
  K^G(S) \ar{r} & TC^G(S)
  \end{tikzcd}
  \]
  is cartesian.
\end{cor}

\item Let $k$ be a field of exponential characteristic $e$ and $R$ a ring of coefficients. Under the assumption that $e$ is invertible in the ring $R$ (e.g. $R = \bbZ[\tfrac{1}{e}]$), the subcategory of geometric motives is generated as a stable subcategory closed under retracts, by the smaller collection
\[
\{ M(X)(q): X\text{ is a smooth, \emph{proper} $k$-scheme}, q \in \bbZ \}.
\]
In this situation, $\DM_{\mrm{gm}}(k, R)$ admits a bounded weight structure and admits a {\bf weight complex functor}
\[
 \DM_{\mrm{gm}}(k, R) \rightarrow K^b(\mbf{Chow}(k,R)),
\]
where the target, the category of bounded complexes of classical Chow motives with coefficients in $R$, $\mbf{Chow}(k,R)$, is also equipped with a canonical weight structure. This functor is also an example of a nilpotent extension of boundedly weighted stable $\infty$-categories. Corollary~\ref{cor:multdgm} then yields:

\begin{cor}[See Example~\ref{exam:chow-weight-complex}]  \label{cor:mot}
  Let $k$ be a field of exponential characteristic $e$ and $R$ a ring of coefficients such that $e$ is invertible in $R$. Then the square
  \[
  \begin{tikzcd}
  K(\DM_{\mrm{gm}}(k,R)) \ar{d} \ar{r} & TC(\DM_{\mrm{gm}}(k,R)) \ar{d}\\
  K(K^b(\mbf{Chow}(k,R))) \ar{r} & TC(K^b(\mbf{Chow}(k,R)))
  \end{tikzcd}
  \]
  is cartesian.
\end{cor}

\end{enumerate}

\sssec{Other applications} In \cite{Hoyois_2019}, Hoyois and Krishna proved that homotopy $K$-theory satisfies cdh descent for a wide class of stacks using equivariant motivic homotopy theory. This generalizes the result of Cisinski in the case of schemes \cite{Cisinski}, which also appealed to motivic homotopy theory. In \cite[Corollary A.5]{Land_2019}, a proof of cdh descent for homotopy $K$-theory was obtained, in the context of schemes, using vastly different methods. They proved that, more generally, any truncating invariant satisfies cdh descent on schemes \cite[Theorem E]{Land_2019}.
Their ideas were further developed in \cite{BKRS} to show that any localizing invariant satisfies more subtle pro-cdh descent for a wide class of stacks.
Combining Theorem~\ref{thm:dgm-weights} and \cite[Theorem~C]{BKRS} we were able to prove that any truncating invariant satisfies cdh descent in the same context, generalizing the work of Hoyois and Krishna. This application is discussed in Section~\ref{sec:apps}.

In a different context, working over the complex numbers, we were able to apply our results to the lattice conjecture in the context of noncommutative Hodge theory as pioneered by Katzarkov, Kontsevich and Pantev \cite{kkp}. This conjecture asks that the topological $K$-theory of a smooth proper dg-category\footnote{Incarnated in this paper as a $\bbC$-linear stable $\infty$-category; see \cite{cohn} for a comparison.}, in the sense of Blanc \cite{blanc}, admits a Hodge structure. We establish this conjecture for a wide class of derived algebraic stacks, based on the corresponding result for classical algebraic stacks due to \cite{dans-hodge} and a recent result of Konovalov \cite{konavalov}. This application is discussed in Section~\ref{app:latt}.

\ssec{Conventions}\label{sec:conv} We use relatively standard terminology on $\infty$-categories following \cite{HTT,HA,SAG}. Additionally, we use the following notation/convention:
\begin{enumerate}
\item By a {\bf category} we really do mean a $1$-category and an $\infty$-category will be called as such.
\item if $\sC$ is an $\infty$-category we write $\sC^{\simeq} \subset \sC$ for its core aka maximal subgroupid.
\item if $\sC$ is a stable $\infty$-category, then for each $x, y \in \sC$ we denote its mapping spectrum by $\mrm{maps}(x,y)$, while $\Maps(x,y)$ denotes the underlying mapping spaces so that $\Omega^{\infty}\mrm{maps}(x,y) \simeq \Maps(x,y)$; the relationship between the notation $\mrm{end}(x)$ and $\mrm{End}(x)$ is analogous.
\item If $\sC$ is an $\infty$-category with finite coproducts, then we write
\[
\Pre_{\Sigma}(\sC):= \Fun^{\times}(\sC^{\op}, \Spc),
\]
for the $\infty$-category of functors which convert finite coproducts to products.
\end{enumerate}

\ssec{Acknowledgements} The first author would like to thank Lars Hesselholt for bringing to his attention the problem of extending the DGM theorem to categories, as well as his guidance over the years. We would further like to thank Ben Antieau, Sveta Makarova, Charanya Ravi and Lior Yanovski for comments on an earlier draft, Mikhail Bondarko for informing us about Drinfeld quotients, Denis-Charles Cisinski for helpful discussions on localization of additive $\infty$-categories, and Andrei Konovalov for discussions on the lattice conjecture. Special thanks must go to the very patient referee who took the painstaking step of dissecting our first draft (with over a hundred comments!) and recommended numerous improvements. The results were obtained with the support of Russian Science Foundation grant 20-41-04401.

\section{On additive $\infty$-categories} \label{sec:additive} In this section, we develop more thoroughly some aspects of the theory of additive $\infty$-categories, such as their localization and idempotent completion. We keep in mind the analogs of these concepts for stable $\infty$-categories and how they interact with localizing invariants. A result which might be of independent interest is a Schwede-Shipley-style recognition principle for additive $\infty$-categories, given in Theorem~\ref{thm:Morita-theory}. This result feeds into the proof of our main result. We finally explain how Bondarko's theory of weight structures builds a ``bridge" between additive and stable $\infty$-categories.

%
%

\ssec{Aspects of additive $\infty$-categories}

Recall that an additive category is, in particular, enriched in the category of abelian groups. Hence,
if $x \in \mrm{Obj}(\sA)$ then $\Hom(x,x)$ is naturally an associative ring with $\circ$ acting as the
multiplication. From this point of view, it is natural to view an additive category as the
generalization of an associative ring. In higher algebra, the analog of an additive category is an
additive $\infty$-category \cite{ggn}, \cite[Appendix C.1.5]{SAG}.

\begin{defn}\label{def:nilp-add}
A {\bf semi-additive $\infty$-category} (also often called {\bf preadditive $\infty$-category}) $\sA$ is a \emph{pointed} $\infty$-category with finite products and coproducts such that for any pair of objects $x, y$, the canonical map
\[
x \coprod y \rightarrow x \times y,
\]
is an equivalence. We write this object as the {\bf sum} $x \oplus y$. The sum admits a {\bf shear map}
\[
s=(\pi_1, \nabla): x \oplus x \rightarrow x \oplus x,
\] where $\pi_1$ is the first projection and $\nabla$ is the fold map. If this map is an equivalence for all $x \in \mrm{Obj}(\sA)$, then we say that $\sA$ is an {\bf additive $\infty$-category}. An {\bf additive functor} of additive $\infty$-categories is a functor $F:\sA \rightarrow \sB$ that preserves zero objects and all sums.
\end{defn}

We refer to \cite[Section 2]{ggn} for a more extensive discussion and the following result which is \cite[Proposition 2.8]{ggn}:

\begin{prop}\label{prop:ggn} Let $\sC$ be an $\infty$-category with finite coproducts and products. Then the following are equivalent:
\begin{enumerate}
\item the $\infty$-category $\sC$ is additive,
\item the homotopy category $h\sC$ is additive,
\item the forgetful functor
\[
\mrm{Gp}_{\mathbb{E}_{\infty}}(\sC) \rightarrow \sC
\]
is an equivalence.
\end{enumerate}
\end{prop}

Our first goal is to develop a theory of exact sequences of additive $\infty$-categories. We will
consider the $\infty$-category
\[
\mbf{Cat}^{\mrm{add}}_{\infty}
\]
of \emph{small} additive $\infty$-categories and additive functors between them.

\sssec{Verdier quotients of additive $\infty$-categories} \label{sec:add-verd} To develop the theory of exact sequences of additive $\infty$-categories we will follow the case of stable $\infty$-categories analyzed in \cite[Section 5]{blumberg2013universal} and \cite[Chapter I]{nikolaus-scholze}. In the $1$-categorical context, this was studied by the the second author and Bondarko in \cite{bondarko-vova}.

Before we proceed, let us gather the necessary ingredients. First, for a small $\infty$-category $\sC$ and  a collection of arrows $W \subset \sC$ we can form the $\infty$-category $\sC[W^{-1}]$ equipped with a functor $\gamma: \sC \rightarrow \sC[W^{-1}]$ which is the universal $\infty$-category under $\sC$ where all elements in $W$ are invertible (more precisely, see the universal property \cite[Definition 7.1.2]{cisinski-rl}). This localization exists and is unique up to equivalence of $\infty$-categories under $\sC$ \cite[Proposition 7.1.3]{cisinski-rl}. Moreover, the functor can be chosen to be identity on objects \cite[Remark 7.1.4]{cisinski-rl}. Furthermore, the functor $h(\sC) \rightarrow h(\sC[W^{-1}])$ witnesses $h(\sC[W^{-1}])$ as a localization of the category $h(\sC)$ in the sense of ordinary category theory (as reviewed, for example, in \cite[Definition 2.2.8]{cisinski-rl}).

Second, we need a passage from ``small categories" to ``big categories." For motivation, let us recall that in the context of stable $\infty$-categories what one might mean by this is the passage to ind-completion. For $\sC, \sD$ stable $\infty$-categories, we write $\Fun^{\mrm{ex}}(\sC, \sD)$ for the $\infty$-category of functors which are {\bf exact}, i.e., preserve finite limits and colimits. If $\sC$ is a small stable $\infty$-category, then the Yoneda functor
\[
\yo:\sC \rightarrow  \Fun^{\mrm{ex}}(\sC^{\op}, \Spt),
\]
identifies the codomain with $\mrm{Ind}(\sC)$; see, for example, \cite[Proposition 3.2]{blumberg2013universal}. The upshot of going to $\mrm{Ind}(\sC)$ is that the resulting $\infty$-category is a stable presentable $\infty$-category. Therefore localizations of this $\infty$-category can often be constructed via adjoint functor theorems \cite[Corollary 5.5.2.9]{HTT}. The theory of localization sequences in algebraic $K$-theory as pioneered by Thomason \cite{TT}, Neeman \cite{Neeman}, and later in \cite{blumberg2013universal} is developed by passing to large $\infty$-categories first before going back to small categories by taking compact objects and the functor
\[
\sC \rightarrow \mrm{Ind}(\sC)^{\omega},
\]
then witnesses the idempotent completion of $\sC$ \cite[Lemma 5.4.2.4]{HTT}.

For additive $\infty$-categories, by passage from the ``small categories" to ``big categories" we mean the passage to Quillen's nonabelian derived categories through which the Yoneda embedding factors:
\[
\yo:\sC \rightarrow \Pre_{\Sigma}(\sC).
\]
In this case, $ \Pre_{\Sigma}(\sC)$ is a prestable presentable $\infty$-category. The basic properties of this construction are summarized in the next lemma.

\begin{lem}\label{lem:basic} Let $\sC$ be a small $\infty$-category with finite coproducts. Consider the Yoneda functor
\[
\yo:\sC \rightarrow \Pre_{\Sigma}(\sC) \subset \Fun(\sC^{\op}, \Spc).
\]
Then:
\begin{enumerate}
\item the $\infty$-category $\sC$ is additive if and only if $\Pre_{\Sigma}(\sC)$ is prestable.
\item in this case, we have an equivalence:
\[
 \Pre_{\Sigma}(\sC) \simeq \Fun^{\times}(\sC^{\op},\Spt_{\geq 0}), and
\]
the $\infty$-category $ \Pre_{\Sigma}(\sC)$ is a presentable prestable $\infty$-category.
\end{enumerate}
\end{lem}
\begin{proof}
See \cite[Proposition~C.1.5.7~and~Remark~C.1.5.9]{SAG}.
\end{proof}

One consequence of this is an enhancement of the Yoneda functor.

\begin{cor}\label{cor:s-yoneda} Let $\sA$ be an additive $\infty$-category. The Yoneda functor $\sA \rightarrow \Pre_{\Sigma}(\sA)$ canonically promotes to a functor
\[
\sA \rightarrow \Fun^{\times}(\sA^{\op}, \Spt),
\]
which is fully faithful.
\end{cor}

\begin{proof} The functor is given by:
\[
\sA \rightarrow  \Pre_{\Sigma}(\sA)\simeq \Fun^{\times}(\sA^{\op},\Spt_{\geq 0}) \hookrightarrow  \Fun^{\times}(\sA^{\op},\Spt),
\]
where the equivalence is given by Lemma~\ref{lem:basic}(2) and the second functor is induced by the inclusion of connective spectra into all spectra. The functor in question is therefore a composite of two fully faithful functors.
\end{proof}
We use notation
\[
 \widehat{\sA}:=\Fun^{\times}(\sA^{\op}, \Spt),
 \]
 and call the resulting $\infty$-category the $\infty$-category of {\bf additive spectral presheaves}.

We now address the theory of Verdier quotients in the context of additive $\infty$-categories using the nonabelian derived category in analogy with the development of Verdier quotients for stable $\infty$-categories using $\mrm{Ind}$-completions as in \cite[Section 5]{blumberg2013universal}; see also \cite[Section I.3]{nikolaus-scholze}.

From hereon, we adopt the following notation. If $\sA$ is an additive $\infty$-category and $x$ is an object of $\sA$, then the functor represented by $x$:
\[
\Maps_{\sA}(-,x): \sA^{\op} \rightarrow \Spc,
\]
is a product-preserving functor and hence an object of $\Pre_{\Sigma}(\sA)$. By Lemma~\ref{lem:basic}(2), this functor
canonically upgrades to a functor into connective spectra. We denote this functor by
\[
\mbf{Maps}_{\sA}(-,x):\sA^{\op} \rightarrow \Spt_{\geq 0},
\]
and set
\[
\mbf{End}_{\sA}(x) := \mbf{Maps}_{\sA}(x,x).
\]
In what follows also note that cofibers of maps of connective spectra can be computed in spectra.

\begin{thm}\label{thm:exist-quot} Let $\sA$ be a small additive $\infty$-category and $\sB \subset \sA$ be an additive full subcategory. Then
\begin{enumerate}
\item Define
\[
W := \{ f \oplus g\colon b\oplus a \rightarrow b' \oplus a' \text{ where } f\colon b \rightarrow b' \in \sB \text{ and } g: a \rightarrow a'\text{ is invertible} \}.
\]
$\sA/\sB := \sA[W^{-1}]$ is an additive $\infty$-category equipped with an additive functor
\[
\overline{(-)}:\sA \rightarrow \sA/\sB
\] which is a bijection on the sets of objects.
\item If $\sE$ is another additive $\infty$-category, then $\overline{(-)}:\sA \rightarrow \sA/\sB$ induces an equivalence between $\Fun^{\times}(\sA/\sB, \sE)$ and the full subcategory of $\Fun^{\times}(\sA, \sE)$ consisting of those additive functors $G:\sA \rightarrow \sE$ satisfying $G|_{\sB} \simeq 0$.
\item We have the following commutative diagram:
\[
\begin{tikzcd}
\sA \ar{r}{\overline{(-)}} \ar{d} & \sA/\sB \ar{d}\\
\Pre_{\Sigma}(\sA) \ar{r}{L} & \Pre_{\Sigma}(\sA/\sB),
\end{tikzcd}
\]
where the bottom horizontal functor is left adjoint to the functor
\[
\mrm{res}_{\mrm{can}} :\Pre_{\Sigma}(\sA/\sB)  \rightarrow\Pre_{\Sigma}(\sA)  \qquad F \mapsto F \circ\overline{(-)}.
\]

\item If $\sB$ is generated under sums and retracts by a single object $z$, then for any
$F \in \Pre_{\Sigma}(\sA)$ we have:
\[
LF(-) \simeq \Cofib(F(z) \otimes_{\mbf{End}_{\sA}(z)} \mbf{Maps}_{\sA}(-,z) \to F(-)).
\]

\item
In general for any
$F \in \Pre_{\Sigma}(\sA)$ there is an equivalence:
\[
LF(-) \simeq \Cofib(\colim\limits_{z_1,\cdots, z_n \in \sB} F(\bigoplus\limits_{i=1}^n z_i) \otimes_{\mbf{End}_{\sA}(\bigoplus\limits_{i=1}^n z_i)} \mbf{Maps}_{\sA}(-,\bigoplus_{i=1}^n z_i) \to F(-)),
\]
where the colimit is taken over the poset of finite subsets $\{z_1,\cdots, z_n\}$ of objects of $\sB$.
\end{enumerate}
\end{thm}

\begin{proof}
To prove assertion (1) it suffices to show that the the functor on homotopy categories
$$h(\sA) \to h(\sA[W^{-1}])$$
is an additive functor of additive categories by Proposition~\ref{prop:ggn}. This follows from \cite[Proposition 2.2.4]{bondarko-vova}.

By the universal property of the localization, $\overline{(-)}:\sA \rightarrow \sA/\sB$ induces an equivalence between $\Fun^{\times}(\sA/\sB, \sE)$ and the full subcategory of $\Fun^{\times}(\sA, \sE)$ consisting of those additive functors $G:\sA \rightarrow \sE$ for which $G(w)$ is an equivalence for all $w \in W$.
By definition of $W$ and additivity of $G$ the latter condition is equivalent to
 $G(g)$ being an equivalence for any $g\colon b \to b' \in \sB$ which is also equivalent to:
 $G(b) \simeq 0$ for all $b\in \sB$. This proves (2).

To prove assertion (3), we start with the following general observations about an additive functor of $\infty$-categories $\pi:\sA_0 \rightarrow \sA_1$. The restriction of a product-preserving functor $\left(\sA_1\right)^{\op} \rightarrow \Spc$ along $\pi$ remains product-preserving since $\pi$ is additive. Therefore we have a commutative diagram
\[
\begin{tikzcd}
\Pre_{\Sigma}(\sA_1)  \ar{r}{\pi^*} \ar[hook]{d} &\Pre_{\Sigma}(\sA_0)  \ar[hook]{d}\\
\Pre(\sA_1) \ar{r} & \Pre(\sA_0),
\end{tikzcd}
\]
where $\pi^*$ is given by precomposing with $\pi$ \cite[Proposition 5.5.8.10(3)]{HTT}. The functor $\pi^*$ admits a left adjoint $\pi_!$ \footnote{Using, for example, the adjoint functor theorem \cite[Theorem 5.5.2.9]{HTT} where the accessibility hypothesis on $\pi^*$ and the presentability hypotheses on the domain and target $\infty$-categories follows from \cite[Proposition 5.5.8.10(1)]{HTT}.} which, by the commutativity of the diagram of right adjoints above, is characterized by the commutativity of the corresponding diagram of left adjoints:
\[
\begin{tikzcd}
\Pre(\sA_0)  \ar{r}{\pi_!} \ar[swap]{d}{L_{\Sigma}} &\Pre(\sA_1)  \ar{d}{L_{\Sigma}}\\
\Pre_{\Sigma}(\sA_0) \ar{r} & \Pre_{\Sigma}(\sA_1).
\end{tikzcd}
\]
Using the universal property of $\Pre_{\Sigma}$ \cite[Proposition 5.5.8.15]{HTT}, we also get that $\pi_!$ is characterized by the commutativity of
\[
\begin{tikzcd}
\sA_0 \ar{r}{\mrm{\pi}} \ar{d} & \sA_1 \ar{d}\\
\Pre_{\Sigma}(\sA_0) \ar{r}{\pi_!} & \Pre_{\Sigma}(\sA_1).
\end{tikzcd}
\]
These observations then apply to the additive functor
\[
\overline{(-)}:\sA \rightarrow \sA/\sB,
\]
which proves the assertion in (3).

To prove the fourth claim denote by $L'(-)$ the functor given by the formula
on the right. There is also a natural transformation
$\id_{\Pre_{\Sigma}(\sA)} \stackrel{\alpha}\to L'.$
It suffices to prove that $L'$ is a localization onto the subcategory of all $G \in \Pre_{\Sigma}(\sA)$ such that $G|_{\sB}=0$.
Note that the two maps
$$F(z) \otimes_{\mbf{End}_{\sA}(z)} \mbf{Maps}_{\sA}(z,z) \otimes_{\mbf{End}_{\sA}(z)} \mbf{Maps}_{\sA}(-,z) \longrightarrow
F(z) \otimes_{\mbf{End}_{\sA}(z)} \mbf{Maps}_{\sA}(-,z)$$
given by
$$
f\otimes e \otimes e' \mapsto fe \otimes e'
\text{ and }
f\otimes e \otimes e' \mapsto f \otimes ee'
$$
are equivalences for any $F$, so by construction we have a commutative diagram in $\Fun^{\times}(\sC^{\op},\Spt_{\geq 0})$:
\[
\begin{tikzcd}
F(z) \otimes_{\mbf{End}_{\sA}(z)} \mbf{Maps}_{\sA}(z,z) \otimes_{\mbf{End}_{\sA}(z)} \mbf{Maps}_{\sA}(-,z) \arrow[d, "\simeq"] \arrow[r, "\simeq"]
& F(z) \otimes_{\mbf{End}_{\sA}(z)} \mbf{Maps}_{\sA}(-,z) \arrow[d]\arrow[r] & 0\arrow[d]
\\
F(z)\otimes_{\mbf{End}_{\sA}(z)} \mbf{Maps}_{\sA}(-,z) \arrow[d]\arrow[r]
& F(-) \arrow[d, "\alpha_F"]\arrow[r,"\alpha_F"]
& L'F(-)\arrow[d, "\alpha_{L'(F)}"]
\\
0\arrow[r]
 &L'F(-)\arrow[r, "L'(\alpha)"]
& L'L'F(-)
\end{tikzcd}
\]
whose rows and columns are cofiber sequences. Hence $L'(\alpha)(F)$ and $\alpha_{L'(F)}$ are equivalences.
Now by \cite[Proposition~5.2.7.4]{HTT} $L'$ is a localization onto its image.
The map
$$
F(z) \otimes_{\mbf{End}_{\sA}(z)} \mbf{Maps}_{\sA}(z,z) \longrightarrow F(z)
$$
is an equivalence, so $L'F(z) = 0$ and hence the image of $L'$ is conatined in the image of $L$.
Moreover, for any $G \in \Pre_{\Sigma}(\sA)$ such that $G|_{\sB}=0$
$$G(z) \otimes_{\mbf{End}_{\sA}(z)} \mbf{Maps}_{\sA}(z,z) = 0,$$
so $G \simeq L'(G)$ is in the image of $L'$.

To prove the last claim it suffices to show that the functor
$$\mrm{res}_{\mrm{can}} \circ L : \Pre_{\Sigma}(\sA) \to \Pre_{\Sigma}(\sA/\sB) \to \Pre_{\Sigma}(\sA),$$
is a colimit of functors
$$
\Pre_{\Sigma}(\sA) \stackrel{L_S}\to \Pre_{\Sigma}(\sA/\langle z_1, \cdots, z_n\rangle) \stackrel{\mrm{res}_{\mrm{can}, S}}\to \Pre_{\Sigma}(\sA),
$$
over all finite subsets $S=\{z_1, \cdots, z_n\} \subset \sB$.
Since for any $z \in \sB$ there exists $S$ such that $L_SF(z) = 0$, the colimit lands in the image of $L$:
\[
L':\PSh_{\Sigma}(\sA) \rightarrow L(\PSh_{\Sigma}(\sA)) \simeq \PSh_{\Sigma}(\sA/\sB);
\]
here the equivalence is due to part (2).

Now, since $\mrm{res}_{\mrm{can}, S}$ is right adjoint to $L_S$, the canonical map
$$\Maps(\mrm{res}_{\mrm{can}, S} \circ L_SF, G) \cong \Maps(F, G)$$
is an equivalence for any $G$ such that $G(z)=0$ for all $z \in S$. In particular, if $G|_{\sB} = 0$, then we have the above equivalence for any $S$, so
$$\Maps(F, G) \cong  \lim_S \Maps(\mrm{res}_{\mrm{can}, S} \circ L_SF, G)  \cong \Maps(L'F, G),$$ since the limit is cofiltered and the diagram is constant.
Now, note that for any such $G$ we also have a canonical equivalence (by part (2) again):
$$\Maps(F, G) \cong \Maps(LF, G),$$
so by Yoneda lemma $LF \cong L'F$.
\end{proof}

\begin{cor}\label{cor:explicitquotcat}
Let $\sA$ be a small additive $\infty$-category and $\sB \subset \sA$ a full additive subcategory.
For any $x,y \in \sA$
\begin{enumerate}
\item
the space $\Maps_{\sA/\sB}(x,y)$ can be computed as the $\Omega^{\infty}$ of the connective spectrum
\[
\Cofib(\colim\limits_{z_1,\cdots, z_n \in \sB} \mbf{Maps}_{\sA}(\bigoplus\limits_{i=1}^n z_i, y) \otimes_{\mbf{End}(\bigoplus\limits_{i=1}^n z_i)} \mbf{Maps}_{\sA}(x,\bigoplus_{i=1}^n z_i) \to \mbf{Maps}_{\sA}(x,y)).
\]
\item The abelian group $\pi_0\Maps_{\sA/\sB}(x,y)$ can be computed as
\[
\Coker(\bigoplus\limits_{z\in \sB} \pi_0\Maps_{\sA}(z,y) \otimes \pi_0\Maps_{\sA}(x,z) \to \pi_0\Maps_{\sA}(x,y)).
\]
\end{enumerate}
\end{cor}
\begin{proof}
The first claim follows directly from Theorem~\ref{thm:exist-quot}(3~and~5) and taking $\Omega^{\infty}$ to obtain the mapping space from the mapping spectrum.
This claim in particular implies that $\pi_0\Maps_{\sA/\sB}(x,y)$ can be computed as the quotient of
$\pi_0\Maps_{\sA}(x,y)$ modulo the subgroup of those morphisms that factor through $z \in \sB$.
This is exactly what formula (2) says.
\end{proof}

\begin{rem}
A very similar formula to the one in Corollary~\ref{cor:explicitquotcat}(1)
was obtained in the setting of dg-quotients of dg-categories in \cite{Drinfeld2002DGQO}.
\end{rem}

\sssec{Idempotents for additive $\infty$-categories} We now discuss the condition of idempotent completeness in the context of additive $\infty$-categories. Recall that we have the free category containing an idempotent, $\mrm{Idem}$, which is a subcategory of $\mrm{Idem}^+$, the free category containing a retraction; these categories are constructed in  \cite[Definition 4.4.5.2]{HTT}. An {\bf idempotent} in an $\infty$-category $\sC$ is a functor $e: \mrm{Idem} \rightarrow \sC$ and we say that $e$ is {\bf effective} if $e$ extends to a functor $\mrm{Idem}^+ \rightarrow \sC$. We say that $\sC$ is {\bf idempotent complete} if every idempotent is effective. Following \cite[Corollary 4.4.5.14]{HTT} the functor $e: \mrm{Idem} \rightarrow \sC$ admits a colimit if and only if $e$ is an effective idempotent.

\begin{rem}
An idempotent $e$ in $\sC$ gives rise to an idempotent of $h(\sC)$ which can be
identified with a morphism $e$ in $h(\sC)$ such that
$e\circ e = e$. In general, this datum is not sufficient for defining an idempotent  \cite[Counterexample~4.4.5.19]{HTT}. However, if $\sC$ is a stable $\infty$-category, then $\sC$ is idempotent complete if and only if $h(\sC)$ is \cite[Lemma 1.2.4.6]{HA} and, in fact, any idempotent in $h(\sC)$ lifts to one in $\sC$ \cite[Warning 1.2.4.8]{HA}. We observe that the same proof also works for additive $\infty$-categories.
\end{rem}

\begin{prop}\label{prop:idempotents_in_add}

For an additive $\infty$-category $\sA$ any idempotent in $h(\sA)$ can be lifted to an
idempotent in $\sA$.
Consequently, the following are equivalent:
\begin{enumerate}
  \item $\sA$ is idempotent-complete
  \item $h(\sA)$ is idempotent-complete
  \item any morphism $x \stackrel{e}\to x \in \sA$ such that there is a homotopy $e\circ e \simeq e$ is equivalent to a morphism of the form
  $$x_1 \oplus x_2 \stackrel{\begin{pmatrix}\id_{x_1} & 0 \\ 0 & 0\end{pmatrix}}\to x_1 \oplus x_2$$
\end{enumerate}
\end{prop}
\begin{proof}
 Let
\[
e: x \rightarrow x,
\]
be an idempotent in $h(\sA)$, i.e., a morphism $e$ equipped with a homotopy $h: e \rightarrow e \circ e$. Since $\sA$ is an additive $\infty$-category,
there is a fully faithful Yoneda functor $\sA \rightarrow \widehat{\sA}$ by Corollary~\ref{cor:s-yoneda}. We consider the morphism $e$ in $\Fun^{\times}(\sA^{\op}, \Spt)$, which defines an idempotent in $h(\widehat{\sA})$.

To prove the result we will promote $e$ to a weak retraction diagram in the sense of \cite[Definition 4.4.5.4(2), Remark 4.4.5.5]{HTT} which we briefly review. Let $\mrm{Ret} \subset \mrm{Idem}^+$ be the simplicial set classifying retractions (as in \cite[Page 304]{HTT}). Then a weak retraction diagram is a functor out of $\mrm{Ret}$. The inclusion $\mrm{Ret} \subset \mrm{Idem}^+$ is an inner anodyne map \cite[Proposition 4.4.5.6]{HTT} and thus for any $\infty$-category $\sC$ the induced functor $\Fun(\mrm{Idem}^+,\sC) \rightarrow \Fun(\mrm{Ret},\sC)$ is trivial fibration of simplicial sets \cite[Corollary 4.4.5.7]{HTT}. Concretely, any weak retraction diagram extends, in an essentially unique way, to a functor out of $\mrm{Idem}^+$. We first claim that the idempotent $e$ in $h(\widehat{\sA})$ extends to a map of simplicial sets $\mrm{Ret} \rightarrow \widehat{\sA}$ and thus to a functor out of $\mrm{Idem}^+$.

The argument is the same as in the second paragraph of the proof of \cite[Lemma 1.2.4.6]{HA} which we sketch for the reader's convenience. Denote by $x_0$ (resp. $x_1$) the sequential colimit
$$x \stackrel{e}\to x\stackrel{e}\to x \to \cdots$$
$$(\text{resp. } x \stackrel{1-e}\to x\stackrel{1-e}\to x \to \cdots),$$
in $\widehat{\sA}$\footnote{We note that the existence of the map $1-e$ is where the additivity of $\widehat{\sA}$ gets used.}. Since the latter $\infty$-category is cocomplete, these colimits exist. Now, for any $F \in \widehat{\sA}$, we note that $e$ induces an idempotent on the abelian group $\pi_iF(x)$ for all $i \in \bbZ$ so that we have a splitting:
\[
\pi_iF(x) \simeq \pi_iF(x)_+ \oplus \pi_iF(x)_-,
\]
where $e$ acts as the identity on $ \pi_iF(x)_+$ and as $0$ on $ \pi_iF(x)_-$. Thus, the tower of abelian groups
\[
\cdots \pi_iF(x) \xrightarrow{e^*} \pi_iF(x) \xrightarrow{e^*} \pi_iF(x),
\]
splits as a direct sum of towers:
\[
\cdots \pi_iF(x)_+ \xrightarrow{\id} \pi_iF(x)_+ \xrightarrow{\id} \pi_iF(x)_+,
\]
and
\[
\cdots \pi_iF(x)_- \xrightarrow{0} \pi_iF(x)_- \xrightarrow{0} \pi_iF(x)_-.
\]
such that
\[
\lim  \pi_iF(x)  \simeq \pi_iF(x)_+,
\]
and the $\lim^1$-term vanishes. This means that the canonical map $x \rightarrow x_0$ induces an isomorphism of abelian groups, for all $i \in \bbZ$:
\[
\pi_iF(x_0) \simeq \pi_i(\mrm{maps}(\colim x, F)) \simeq  \lim \pi_iF(x) \simeq  \pi_iF(x)_+.
\]
By a similar argument for the idempotent $1-e$, we have
\[
\pi_iF(x_1) \simeq \pi_iF(x)_-.
\]
Therefore the canonical map $x \rightarrow x_0 \oplus x_1$ defines an isomorphism for all $F$ and all $i \in \bbZ$:
\[
\pi_iF(x) \cong \pi_iF(x_0) \oplus \pi_iF(x_1),
\]
and thus induces a splitting:
\[
F(x) \simeq F(x_0) \oplus F(x_1).
\]
Since this is true for all $F$, we also have a splitting of representable presheaves:
\[
x \simeq x_0 \oplus x_1.
\]
Furthermore, we see that the inclusion of summand, $x_0 \rightarrow x$, together with the canonical map to the colimit, $x \rightarrow x_0$, induces a weak retraction diagram
\[
\begin{tikzcd}
 & x \ar{dr} & \\
 x_0 \ar{ur} \ar{rr}{\id} &  & x_0,
\end{tikzcd}
\]
in $\widehat{\sA}$ as claimed. To see that the idempotent $e$ in $h(\sA)$ lifts to one on $\sA$, observe that the corresponding functor $\mrm{Idem}^+ \rightarrow \widehat{\sA}$ restricts to a functor landing in the Yoneda image:
\[
\mrm{Idem} \rightarrow \sA \subset \widehat{\sA};
\]
note that we have used the additivity of $\sA$ to get the fully faithfulness of the Yoneda functor; see Corollary~\ref{cor:s-yoneda}.

We now prove that the three assumptions are equivalent. The third assumption is clearly equivalent to
the second assumption.
Since any idempotent in $h(\sA)$ lifts to an idempotent in $\sA$, it must be effective if the lift
is effective in
$\sA$, so we have $(1) \Rightarrow (2)$.
Now we prove $(2) \Rightarrow (1)$. Let $e \colon \on{Idem} \to \sA$ be an idempotent in $\sA$.
Denote by $x$ the colimit of $e$ in $\widehat{\sA}$.
It is a retract of an object of $\sA$, so by the assumption its image in
$h(\widehat{\sA})$ is isomorphic to an object of $h(\sA)$ considered as a full
subcategory of $h(\widehat{\sA})$. So the colimit of $e$ is equivalent to an object of
$\sA$ and $e$ is effective in $\sA$.
\end{proof}

Proposition~\ref{prop:idempotents_in_add} justifies the usage of some standard terminology from the theory of additive categories in the context of additive $\infty$-categories.

\begin{defn}\label{defn:idem}

Let $\sA$ be an additive $\infty$-category.

\begin{enumerate}
\item
We say that an idempotent $e: x \rightarrow x$ in $\sA$ {\bf splits} if it is homotopic to
a morphism of the form
\[
x_1 \oplus x_2 \xrightarrow{\begin{pmatrix}
\id_{x_1} & 0 \\
0 & 0 \\
\end{pmatrix}} x_1 \oplus x_2.
\]

\item
A full subcategory $\sA \subset \sB$ is said to be {\bf retract-closed} if every idempotent of an object of
$\sA$ that splits in $\sB$ also splits in $\sA$.

\item
 If $\sA$ is an additive $\infty$-category, then we say that $\sA$ is {\bf idempotent complete} (or {\bf absolutely Karoubi closed})
 if each idempotent $e: x \rightarrow x$ in $\sA$ splits.
\end{enumerate}
\end{defn}

\begin{lem}\label{lem:kar} Let $\sA$ be a small additive $\infty$-category. Then there exists a small additive $\infty$-category $\mrm{Kar}(\sA)$ and an additive functor $\sA \rightarrow \mrm{Kar}(\sA)$ which is an initial functor among all idempotent complete additive $\infty$-categories receiving an additive functor from $\sA$. We call this category the idempotent completion of $\sA$.
\end{lem}

\begin{proof} According to \cite[Proposition 5.1.4.2]{HTT} and the identifications of Proposition~\ref{prop:idempotents_in_add}, we can take the embedding
\[
\sA \rightarrow \Pre_{\Sigma}(\sA),
\]
and set $\mrm{Kar}(\sA)$ to be the full subcategory of $\Pre_{\Sigma}(\sA)$ spanned by functors which are retracts of the Yoneda image. Note that this is also an additive $\infty$-category.
\end{proof}

\begin{defn}
If $f:\sA \rightarrow \sB$ is an additive functor between small additive $\infty$-categories, then we get a functor $\mrm{Kar}(f):\mrm{Kar}(\sA) \rightarrow \mrm{Kar}(\sB)$. We say that that $f$ is an {\bf equivalence up to idempotent completion} if $\mrm{Kar}(f)$ is an equivalence. We will also denote the subcategory of additive $\infty$-categories spanned by those which are idempotent complete as:
\[
\mbf{Cat}^{\mrm{Kar}}_{\infty} \subset \mbf{Cat}^{\mrm{add}}_{\infty}.
\]
\end{defn}

\begin{defn}\label{def:quotcat}
Let $\sA, \sC$ be small additive $\infty$-categories and $\sB \subset \sA$ be an additive full subcategory. We say that the sequence
$$\sB \to \sA \to \sC$$
is an {\bf exact sequence of additive $\infty$-categories} if (1) the composite is null and (2) the induced map $\sA/\sB \rightarrow \sC$ is an equivalence up to idempotent completion.
\end{defn}

\sssec{Comparison with exact sequences in stable $\infty$-categories}\label{sect:comparestab} We now explain the passage of going from small additive $\infty$-categories to small stable $\infty$-categories. We denote by $\mbf{Cat}^{\mrm{st}}_{\infty}$ the $\infty$-category of small stable $\infty$-categories and exact functors between them. From an additive $\infty$-category we can extract an object of $\mbf{Cat}^{\mrm{st}}_{\infty}$ by the next procedure.

\begin{defn} Let $\sA$ be a small additive $\infty$-category. The small stable $\infty$-category of {\bf finite cell $\sA$-modules} is the smallest stable subcategory of $\widehat{\sA}$ containing the Yoneda image:
\[
\sA^{\mrm{fin}} \subset \widehat{\sA}.
\]
\end{defn}

\begin{rem}\label{rem:sw-pshfin}
The $\infty$-category $\sA^{\mrm{fin}}$ is equivalent to the Spanier-Whitehead stabilization of the $\infty$-category $\PSh^{\fin}_{\Sigma}(\sA)$ --- the smallest subcategory of
$\PSh_{\Sigma}(\sA)$ containing all representable presheaves and closed under finite colimits.
Indeed, $SW(\PSh_{\Sigma}^\mrm{fin}(\sA))$ is a full subcategory of $\widehat{\sA}$ (this follows from \cite[Remark~C.1.1.6]{SAG}) and the image coincides with the smallest subcategory of $\widehat{\sA}$ containing representable presheaves and closed under taking shifts and finite colimits.

Now combining \cite[Proposition 5.3.6.2]{HTT} and \cite[Proposition~C.1.1.7]{SAG} we get a convenient universal property for $\sA^{\fin}:$
$$\Fun^{\times}(\sA, \sC) \cong \Fun^{\mrm{ex}}(\sA^{\fin}, \sC)$$
for any stable $\infty$-category $\sC$.
In particular, the functor
\begin{equation}\label{eq:funfin}
(-)^{\fin}:\mbf{Cat}^{\mrm{add}}_{\infty} \to \mbf{Cat}^{\mrm{st}}_{\infty},
\end{equation}
is left adjoint to the forgetful functor. Hence we should regard $\sA^{\fin}$ as the ``free stable $\infty$-category" generated by $\sA$.
\end{rem}

\begin{exam}\label{exam:proj-perf} Let $R$ be a connective $\bbE_1$-ring spectrum. In the $\infty$-category of connective $R$-module spectra $ \mbf{RMod}_R^{\mrm{cn}}$, we have the subcategory of {\bf projective $R$-modules}
\[
\mbf{Proj}_R \subset \mbf{RMod}_R^{\mrm{cn}},
\]
defined as the subcategory of projective objects in the sense of \cite[Definition 5.5.8.18]{HTT}. These objects can also be characterized as retracts of free $R$-modules \cite[Proposition 7.2.2.7]{HA}. Inside $\mbf{Proj}_R$ there is the subcategory of finite generated ones
\[
\mbf{Proj}_R^{\mrm{fg}} \subset \mbf{Proj}_R,
\]
which can be characterized as those with $\pi_0$ being finitely generated as a $\pi_0R$-module \cite[Corollary 7.2.2.9]{HA}.
The $\infty$-category $\mbf{Proj}^{\mrm{fg}}_R$ is small and idempotent complete and
we have an equivalence
\[
(\mbf{Proj}^{\mrm{fg}}_R)^{\mrm{fin}} \simeq \mbf{Perf}_R.
\]


\end{exam}

Let us now recall the theory of exact sequences in small stable $\infty$-categories. We denote, following our notation for additive $\infty$-categories, by $\mrm{Kar}(\sC)$ the idempotent completion of a stable $\infty$-category $\sC$ (see \cite[Section 2.2]{blumberg2013universal} for a treatment in this context where they write $\mrm{Idem}(\sC)$ for $\mrm{Kar}(\sC)$). Suppose that we have a sequence of functors
\[
\sB \rightarrow \sA \rightarrow \sC
\]
in $\mbf{Cat}^{\mrm{st}}_{\infty}$. Then we say it is {\bf exact} if the composite is zero, $\sB  \rightarrow \sA$ is fully faithful and the functor
\[
\sA/\sB   \rightarrow \sC
\]
becomes an equivalence after idempotent completion.

\begin{lem}\label{lem:comparing_cofibers}
  The functor~\eqref{eq:funfin} sends exact sequences of additive $\infty$-categories to
  exact sequences of stable $\infty$-categories.
\end{lem}
\begin{proof}
Suppose that we have an exact sequence of additive $\infty$-categories
$$\sB \to \sA \to \sC.$$ We consider the diagram in $\mbf{Cat}^{\mrm{st}}_{\infty}$:
$$\sB^{\fin} \to \sA^{\fin} \to \sC^{\fin}.$$

Fully faithfulness of $\sB^{\fin} \to \sA^{\fin}$ follows, for example, from \cite[Proposition~5.3.5.11]{HTT}.
It follows from Theorem~\ref{thm:exist-quot}(2) that $\sA \to \sA/\sB$ can be identified with the cofiber of the map
$\sB \to \sA$ in $\mbf{Cat}^{\mrm{add}}_{\infty}$. As discussed in Remark~\ref{rem:sw-pshfin} $(-)^{\fin}$ is a left adjoint,
so it preserves cofibers.
Lastly, note that
\[
\mrm{Kar}((-)^{\fin}) \simeq (\mrm{Kar}(-))^{\fin},
\]
since both functors are left adjoint to the forgetful functor
$\mbf{Cat}^{\mrm{perf}}_{\infty} \to \mbf{Cat}^{\mrm{add}}_{\infty}$. 

Therefore, we conclude that the functor
$$\mrm{Kar}(\sA^{\fin}/\sB^{\fin})\simeq \mrm{Kar}((\sA/\sB)^{\fin}) \to \mrm{Kar}(\sC^{\fin})$$
is indeed an equivalence.
\end{proof}

\sssec{Morita theory for additive $\infty$-categories}\label{ssec:morita} Later in the paper we will frequently reduce questions about certain additive $\infty$-categories to those of the form $\mbf{Proj}^{\mrm{fg}}_R$.
The next result is the main tool that allows us to do so. This is an analog of the
well-known result of Schwede and Shipley \cite[Theorem~3.1.1]{SchwedeShipley2003} (see also \cite[Section 7.1.2]{HA})
in the context of additive $\infty$-categories. The additive $\infty$-category of projective modules over the sphere spectrum $\mbf{Proj}^{\mrm{fg}}_{\mathbb{S}}$ is an idempotent complete additive $\infty$-category and thus is an object of $\mbf{Cat}^{\mrm{Kar}}_{\infty}$. We denote the slice category under $\mbf{Proj}^{\mrm{fg}}_{\mathbb{S}}$ by:
\[
\mbf{Cat}^{\mrm{Kar}}_{\infty, \mbf{Proj}^{\mrm{fg}}_\mathbb{S}/}.
\]
Its objects are additive functors $\mbf{Proj}^{\mrm{fg}}_{\mathbb{S}} \rightarrow \sA$; equivalently this picks an object $R = F(\bbS)$ of $\sA$. Hence we can regard $\mbf{Cat}^{\mrm{Kar}}_{\infty, \mbf{Proj}^{\mrm{fg}}_\mathbb{S}/}$ as the $\infty$-category of idempotent complete additive $\infty$-categories with a chosen object.

\begin{thm}\label{thm:Morita-theory}
Denote by $\Theta$ the functor
\[
\mrm{Alg}_{\mathbb{E}_1}(\Spt^\mrm{cn}) \to \mbf{Cat}^{\mrm{Kar}}_{\infty, \mbf{Proj}^{\mrm{fg}}_\mathbb{S}/ }
\]
that sends a ring $R$ to the object in the undercategory:
\[
\mbf{Proj}_\mathbb{S}^{\mrm{fg}}\to \mbf{Proj}_R^{\mrm{fg}}.
\]
Then,
\begin{enumerate}
\item
the functor $\Theta$ is fully faithful; 

\item
an object $\mbf{Proj}^{\mrm{fg}}_\mathbb{S} \stackrel{F}\to \sA$ in $\mbf{Cat}^{\mrm{Kar}}_{\infty, \mbf{Proj}^{\mrm{fg}}_\mathbb{S}/ }$ is in
the essential image of $\Theta$ if and only if
$\sA$ is generated by $F(\mathbb{S})$ under finite sums and retracts. In this case, $F$ is equivalent to $\Theta(\mbf{End}_\sA(F(\mathbb{S})))$. In other words, $F$ is equivalent to
\[
\mbf{Proj}_\mathbb{S}^{\mrm{fg}}\to \mbf{Proj}_{\mbf{End}_\sA(F(\mathbb{S}))}^{\mrm{fg}},
\]
given by extension along the map of conenctive $\bbE_1$-ring spectra: $\mathbb{S} \rightarrow \mbf{End}_\sA(F(\mathbb{S}))$.
\end{enumerate}
\end{thm}
\begin{proof}
To prove $(1)$ note that the canonical functor $\widehat{\mbf{Proj}^{\mrm{fg}}_R} \to \mbf{RMod}_R$ is 
an equivalence, so the composite functor
\[
\mrm{Alg}_{\mathbb{E}_1}(\Spt^\mrm{cn}) \stackrel{\Theta}\to \mbf{Cat}^{\mrm{Kar}}_{\infty, \mbf{Proj}^{\mrm{fg}}_\mathbb{S}/ } \stackrel{\widehat{(-)}}\to \mathcal{P}\mrm{r}^\mrm{St}_{\Spt/ },
\]
is equivalent to the functor sending
$R$ to $\mbf{RMod}_R$.
Combining \cite[Theorem~4.8.5.5]{HA} and \cite[Proposition~4.8.2.18]{HA} we see that this functor is
fully faithful.
Hence to prove that $\Theta$ is fully faithful it suffices to prove that
the induced map on mapping spaces
\[
\Maps_{\mbf{Cat}^{\mrm{Kar}}_{\infty, \mbf{Proj}^{\mrm{fg}}_\mathbb{S}/}}(\mbf{Proj}_R^{\mrm{fg}}, \mbf{Proj}_S^{\mrm{fg}}) \rightarrow \Maps_{\mathcal{P}\mrm{r}^\mrm{St}_{\Spt/ } }(\mbf{RMod}_R, \mbf{RMod}_S)
\]
is an equivalence.
This map is an epimorphism since the functor $\widehat{(\Theta(-))}$ was already known to be fully faithful; it suffices to prove that this map is a subspace inclusion. The above map fits as the upper horizontal map in the following diagram where the vertical arrows define fiber sequences (from how mapping spaces of undercategories are computed):
\[
\begin{tikzcd}
\Maps_{\mbf{Cat}^{\mrm{Kar}}_{\infty, \mbf{Proj}^{\mrm{fg}}_\mathbb{S}/}}(\mbf{Proj}_R^{\mrm{fg}}, \mbf{Proj}_S^{\mrm{fg}}) \arrow[r]\arrow[d] & \Maps_{\mathcal{P}\mrm{r}^\mrm{St}_{\Spt/ } }(\mbf{RMod}_R, \mbf{RMod}_S)\arrow[d]\\
\Maps_{\mbf{Cat}^{\mrm{Kar}}_{\infty}}(\mbf{Proj}_R^{\mrm{fg}}, \mbf{Proj}_S^{\mrm{fg}}) \arrow[r]\arrow[d] & \Maps_{\mathcal{P}\mrm{r}^\mrm{St}}(\mbf{RMod}_R, \mbf{RMod}_S)\arrow[d]\\
\Maps_{\mbf{Cat}^{\mrm{Kar}}_{\infty}}(\mbf{Proj}_\mathbb{S}^{\mrm{fg}}, \mbf{Proj}_S^{\mrm{fg}}) \arrow[r] & \Maps_{\mathcal{P}\mrm{r}^\mrm{St}}(\Spt, \mbf{RMod}_S)
\end{tikzcd}
\]
Now, combining the universal properties \cite[Proposition~5.3.6.2]{HA} and \cite[Proposition~C.1.1.7]{SAG} we obtain that the middle and the bottom horizontal functors are subspace inclusions.
Therefore, we conclude that the upper horizontal functor is also a subspace inclusion.



Now we prove $(2)$. If $\sA=\mbf{Proj}_R^{\mrm{fg}}$ for a connective $\mathbb{E}_1$-ring $R$, then $R=F(\mathbb{S})$
generates $\mbf{Proj}_R^{\mrm{fg}}$ under finite sums
and retracts by definition of this category (see Example~\ref{exam:proj-perf}). It suffices to prove the
converse implication. Assume that $\sA$ is generated, as an additive $\infty$-category, by $F(\mathbb{S})$. Let $R:=\mbf{End}(F(\mathbb{S}))$ be the endomorphism $\bbE_1$-ring spectrum. We have a functor
\[
\mbf{Maps}(F(\mathbb{S}),-): \sA \rightarrow \mbf{RMod}_R
\]
which is fully faithful since it is fully faithful on $F(\mathbb{S})$ and $F(\mathbb{S})$ generates $\sA$ under direct sums and retracts. To conclude, we note that functor sends $F(\mathbb{S})$ to $R$ and its essential image consists precisely of finite direct sums of $R$ and retracts, i.e., the subcategory $\mbf{Proj}^\mrm{fg}_R$.

\end{proof}

\ssec{From additive $\infty$-categories to stable $\infty$-categories via weights}\label{sec:additive-to-weights} Our goal in this section is to explain how one can recognize objects in the essential image of the functor~\eqref{eq:funfin}. This is done via Bondarko's theory of weight structures, which we recall for  the reader's convenience.

\begin{defn} \label{def:weight} A {\bf weight structure} on a stable $\infty$-category $\sC$ is the data of two retract-closed subcategories $(\sC_{w\geq 0}, \sC_{w\leq 0})$ such that:
\begin{enumerate}
\item $\Sigma\sC_{w \geq 0} \subset \sC_{w \geq 0}$, $\Omega\sC_{w \leq 0} \subset \sC_{w \leq 0}$; write
\[
\sC_{w \geq n} := \Sigma^n\sC_{w \geq 0} \qquad \sC_{w \leq k} := \Sigma^k\sC_{w \leq 0}.
\]
\item if $x \in \sC_{w \leq 0}, y \in \sC_{w \geq 1}$ then
\[
\pi_0\Maps_{\sC}(x, y) \simeq 0,
\]
\item for any object $x \in \sC$ we have a cofiber sequence
\begin{equation}\label{eq:cofib}
x_{\leq 0} \rightarrow x \rightarrow x_{\geq 1},
\end{equation}
where $x_{\leq 0} \in \sC_{w \leq 0}$ and $x_{\geq 1} \in  \sC_{w \geq 1}.$ We call these the {\bf weight truncations} of $x$.
\end{enumerate} We say that the weight structure is {\bf bounded} if \[
\sC = \bigcup\limits_n \left(\sC_{w\ge -n}\cap \sC_{w\le n}\right).
\]
A stable $\infty$-category $\sC$ equipped with a (bounded) weight structure is called a  ({\bf boundedly}) {\bf weighted stable $\infty$-category}. We typically write
\[
(\sC, w),
\]
for a boundedly weighted stable $\infty$-category.
\end{defn}

\begin{defn} Let $(\sC,w), (\sD, w')$ be weighted stable $\infty$-categories and let $f: \sC \rightarrow \sD$ be an exact functor of the underlying stable $\infty$-categories. We say that $f$ is {\bf weight exact} if $f(\sC_{w \geq 0}) \subset \sD_{w' \geq 0}$ and $f(\sC_{w \leq 0}) \subset \sD_{w' \leq 0}$. In this case we write
\[
f:(\sC,w) \rightarrow (\sD, w').
\]
\end{defn}

We denote by
\[
\mbf{WCat}^{\mrm{st},b}_{\infty},
\]
the $\infty$-category of small, boundedly weighted stable $\infty$-categories and weight exact functors. Formally this can be defined as the full subcategory of the pullback $\infty$-category:
\[
\Maps(\{\bullet_{\le 0} \to \bullet \leftarrow \bullet_{\ge 0}\}, \mbf{Cat}_{\infty}) \times_{\mbf{Cat}_{\infty}} \mbf{Cat}_{\infty}^{\mrm{st}},
\]
spanned by the triples $(F, F(\bullet), F(\bullet) \stackrel{\id}\to F(\bullet))$ such that the morphisms
$F(\bullet_{\le 0}) \to F(\bullet)$ and $F(\bullet_{\ge 0}) \to F(\bullet)$
are embeddings of full subcategories satisfying the axioms of a weight structure.

\begin{rem}\label{rem:t} The axioms of a weight structure can be thought of as ``dual" to those of a $t$-structure. However, weight structures differ from $t$-structures in at least two important aspects (and definitely more!):
\begin{enumerate}
\item First, the cofiber sequence~\eqref{eq:cofib} is \emph{not} functorial in $x$. Indeed, in the case of $\mbf{K}(\sA)$, the $\infty$-category obtained from complexes in an additive category $\sA$ by inverting homotopy equivalences, weight truncations correspond to the brutal/stupid truncation of complexes (\cite[Remark~1.2.3(1)]{weight-examples}). Morally speaking, while the notion of a $t$-structure abstracts the Postnikov tower (say in the $\infty$-category of spectra or in the derived $\infty$-category of a ring), the notion of a weight structure abstracts the cellular tower. The category $\sC_{w \geq 0}$ should be thought of as the category of ``cells of non-negative dimensions" and conversely,
$\sC_{w \leq 0}$ are ``cells of non-positive dimensions".
Axiom (2) abstracts the fact that there should be no nontrivial maps from a lower-dimensional cell to a higher-dimensional cell, just as there are no nontrivial maps from a lower-dimensional sphere to a higher dimensional sphere. The non-functoriality of weight truncations is related to the fact that a cellular presentation of an object is a choice.

\item Secondly, we define the {\bf weight heart} of a weighted $\infty$-category as
\[
\sC^{\heartsuit_w}:= \sC_{w\geq 0} \cap \sC_{w\leq 0}.
\]
For $t$-structures, the heart is always an abelian category while for weight structures, the heart is an additive $\infty$-category. \end{enumerate}
\end{rem}

We refer the reader to next few sections for examples, while we make further remarks.

\begin{rem} \label{rem:extension} We also note that besides being retract-closed, $\sC_{w \geq 0}$ and $\sC_{w \leq 0}$ are also extension stable: if $x \rightarrow y \rightarrow z$ is a cofiber sequence in $\sC$ where $x, z \in \sC_{w \geq 0}$ then so is $y$. The same goes for $\sC_{w \leq 0}$; we refer to \cite[Proposition 1.3.3(3)]{bondarko-weights} for a proof.
\end{rem}

\begin{rem} \label{rem:constructing_weights}
Bounded weight structures are quite easy to construct. Let $\sC$ be a stable $\infty$-category. Assume we are given with a subcategory $N \subset \sC$ such that
\begin{enumerate}
  \item $N$ generates $\sC$ under finite limits, finite colimits and retracts,
  \item $N$ is {\bf negatively self-orthogonal}, that is, for any $x, y \in N$ the mapping spectrum
  \[
  \mrm{maps}_{\sC}(x, y),
  \] is connective. Equivalently,
  \[
  \pi_0\Maps_{\sC}(x,\Sigma^n y) = 0,
  \] for $n>0$.
\end{enumerate}
According to \cite[Corollary~2.1.2]{weight-examples}, defining
$$\sC_{w\ge 0} = \{\text{retracts of finite colimits of objects of } N\}$$
$$\sC_{w\le 0} = \{\text{retracts of finite limits of objects of } N\}$$
gives a weight structure on $\sC$ whose heart is the minimal retract-closed additive subcategory
containing $N$.
Note that it is a bounded weight structure. Indeed, the subcategory
$$\bigcup\limits_n \left(\sC_{\ge -n}\cap \sC_{\le n}\right) \subset \sC$$
contains $N$ and is closed under finite limits, finite colimits and
retracts, so by assumption (1) the inclusion is an equality.
\end{rem}


In order to use weight structures to detect the image of $(-)^{\fin}$, we factor it through $\mbf{WCat}^{\mrm{st},b}_{\infty}$.

\begin{constr} We construct a factorization:
\begin{equation}\label{eq:wcat}
\begin{tikzcd}
 &\mbf{WCat}^{\mrm{st},b}_{\infty} \ar{d} \\
\mbf{Cat}^{\mrm{add}}_{\infty} \ar[swap]{r}{(-)^{\fin}} \ar[dashed]{ur}{((-)^{\fin},w)}& \mbf{Cat}_{\infty}^{\mrm{st}}.\\
\end{tikzcd}
\end{equation}
Here, the right vertical map is the forgetful functor. To construct the dashed arrow, it suffices to equip $\sA^{\fin}$ with a canonical bounded weight structure. Indeed, the construction follows because a functor of boundedly weighted stable $\infty$-categories is weight exact if and only if it preserves the heart \cite[Remark~1.13]{vova-negative}.  To do so, we use Remark~\ref{rem:constructing_weights}. Let $N$ be the essential image of $\sA \rightarrow \sA^{\fin}$. By definition, $N$ generates $\sA^{\fin}$ under finite limits, finite colimits and retracts. To check the self-orthogonality condition: for $n > 0$ we have for any $x, y$ in $\sA$ (which we abusively identify with the essential image of $\sA \rightarrow \sA^{\fin}$):
\[
\pi_0(\mrm{maps}_{\sA^{\fin}}(x, \Sigma^n y)) \simeq \pi_0(\Sigma^n \mrm{maps}_{\sA^{\fin}}(x, y)) \simeq \pi_{-n}(\mbf{Maps}_{\sA}(x, y)) \simeq 0.
\]
Here we use the convention of Theorem~\ref{thm:exist-quot} where $\mbf{Maps}(-,y)$ indicate a functor landing in connective spectra; this also explains the last equivalence.
\end{constr}

The next theorem is basically a restatement of a result of the second author's \cite[Corollary 3.4]{vova-negative}. To state it, we note that a boundedly weighted stable $\infty$-category is idempotent complete if its underlying stable $\infty$-category is.

\begin{thm}\label{thm:vova} We have an adjoint pair
\[
((-)^{\fin},w): \mbf{Cat}^{\mrm{add}}_{\infty} \rightleftarrows \mbf{WCat}^{\mrm{st},b}_{\infty}: (-)^{\heartsuit_w}.
\]
Furthermore:
\begin{enumerate}
\item the right adjoint functor $(-)^{\heartsuit_w}$ is fully faithful,
\item the adjoint pair restricts to an equivalence of $\infty$-categories of idempotent complete $\infty$-categories on both sides.
\end{enumerate}
\end{thm}

\begin{proof}
By Remark~\ref{rem:sw-pshfin} we have a binatural equivalence
\begin{equation} \label{eq:binatural}
\Fun^{\times}(\sA, \sC) \cong \Fun^{\mrm{ex}}(\sA^{\fin}, \sC)
\end{equation}
for any additive $\infty$-category $\sA$ and stable $\infty$-category $\sC$.
Now suppose $\sC$ is equipped with a bounded weight structure $u$. Then,
\begin{enumerate}
\item we have a subspace inclusion
\[
\Maps_{\mbf{Cat}^{\mrm{add}}_{\infty}}(\sA, \sC^{\heartsuit_u}) \subset \Fun^{\times}(\sA, \sC),
\] given by additive functors $\sA \rightarrow \sC$ which land in $\sC^{\heartsuit_u}$;
\item we have a subspace inclusion
\[
\Maps_{\mbf{WCat}^{\mrm{st},b}_{\infty}}((\sA^{\fin},w), (\sC,u)) \subset \Fun^{\mrm{ex}}(\sA^{\fin}, \sC),
\]
given by those exact functors which are furthermore weight exact.
\end{enumerate}
Under the equivalence of~\eqref{eq:binatural}, these subspaces are identified:
\[
\Maps_{\mbf{Cat}^{\mrm{add}}_{\infty}}(\sA, \sC^{\heartsuit_u}) \cong \Maps_{\mbf{WCat}^{\mrm{st},b}_{\infty}}((\sA^{\fin},w), (\sC,u)).
\]
which gives us the desired adjunction. The second part now follows from \cite[Corollary 3.4]{vova-negative}.

%

%
\end{proof}


\sssec{Motivic examples} \label{sec:motivic} In this section, we discuss examples of boundedly weighted stable $\infty$-categories coming from the theory of motives. We fix a field $k$ and let $R$ be a ring of coefficients. Denote by $\mbf{DM}(k,R)$ the stable $\infty$-category of motives in the sense of Voevodsky, with coefficients in $R$ (see \cite{vv-cat} for the original reference, \cite{mvw} for a textbook treatment and \cite[Chapter 14]{norms} or \cite{ElmantoKolderup} for $\infty$-categorical treatments). This $\infty$-category admits a functor
\[
M:\Sm_k \rightarrow \mbf{DM}(k,R),
\]
which associates to a smooth $k$-scheme $X$, its motive $M(X)$ and converts products of schemes into tensor products of motives; the coefficient ring will always be clear from the context. Inside $\mbf{DM}(k,R)$ we have the subcategory
\[
\DM_{\mrm{gm}}(k, R) \subset \mbf{DM}(k,R)
\]
of {\bf geometric motives} \cite[Lecture 14]{mvw}. In the language of this paper, the latter is a stable subcategory of $\mbf{DM}(k,R)$, closed under retracts, generated by objects
\[
\{ M(X)(q): X\text{ is a smooth $k$-scheme}, q \in \bbZ \}.
\]

\begin{defn} \label{def:chow} Let $k$ be a field and $R$ a ring of coefficients. The additive $\infty$-category of {\bf Chow motives} (with coefficients in $R$):
\[
\mbf{Chow}_{\infty}(k, R) \subset \DM_{\mrm{gm}}(k,R),
\]
is the smallest additive $\infty$-category generated by
\[
\{ M(X)(q)[2q] : \text{$X$ is smooth and projective over $k$}, q \in \bbZ \},
\]
and retracts thereof.

\end{defn}

We note that the mapping 
spaces in $\mbf{Chow}_{\infty}(k,R)$ are not discrete. Indeed, suppose that $X$ is a smooth $k$-scheme and $Y$ is a smooth projective $k$-scheme which is of pure dimension $d$. Then, if $k$ is a \emph{perfect field}, Friedlander-Voevodsky/Atiyah duality (\cite{Friedlander:2000,Riou}) gives us a computation of the dual in $\mbf{DM}(k;R)$
\[
M(Y)^{\vee} \cong M(Y)(-d)[-2d].
\]
Therefore we have the following computation for all $j \geq 0$:

\begin{eqnarray*}
\pi_j\Maps(M(X)(n)[2n], M(Y)(m)[2m]) & \cong & \pi_0\Maps(M(X), M(Y)(m-n)[2m-2n-j])\\
& \cong & \pi_0\Maps(M(X) \otimes M(Y), M(k)(m-n+d)[2m-2n-j+2d])\\
& \cong &  \pi_0\Maps(M(X \times Y), M(k)(m-n+d)[2m-2n-j+2d])\\
& \cong & H_{\mot}^{2(m+d-n)-j, m+d-n}(X \times Y)\\
& \cong & CH^{m+d-n}(X \times Y; j).
\end{eqnarray*}
This can be non-zero for $j>0$. This calculation leads to the following lemma.

\begin{lem} \label{lem:connective} Let $k$ be a perfect field, $X$ be a smooth $k$-scheme, $Y$ a smooth projective $k$-scheme of dimension $d$ and $n, m \in \bbZ$. Then, for any coefficient ring $R$, the mapping spectrum
\[
\mrm{maps}_{\DM_{\mrm{gm}}(k; R)}(M(X)(n)[2n], M(Y)(m)[2m])
\]
is connective.
\end{lem}

\begin{proof}
By the computation above, we need to show that the group
\[
\pi_j\mrm{maps}(M(X)(n)[2n], M(Y)(m)[2m]) \cong H_{\mot}^{2(m+d-n)-j, m+d-n}(X \times Y;R),
\]
is zero whenever $j < 0$. Without loss of generality, we might as well assume that $R= \bbZ$. By way of comparison with Bloch's higher Chow groups \cite{Voevodsky:2002b}:
\[
H^{p,q}_{\mrm{mot}}(X; \bbZ) \cong \mrm{CH}^q(X, 2q - p),
\]
we see that the desired vanishing is true since $\mrm{CH}^q(X, i)$ is zero whenever $i < 0$.
\end{proof}

The following is a result of Bondarko which we give a proof of for the reader's convenience. We remark that, in what follows, $R$ must be a ring in which the characteristic of $k$ is invertible (e.g. $\bbZ[\tfrac{1}{e}]$). This lets us assume that $k$ is perfect and, more importantly, allows us to use smooth \emph{projective} $k$-schemes as generators for geometric motives.

\begin{thm}[\cite{bondarko-weights, bondarko-resol}] Let $k$ be a field and suppose that $e$ is the exponential characteristic of $k$. Assume that $R$ is a ring where $e$ is invertible. Then there exists a bounded weight structure on $\DM_{\mrm{gm}}(k;R)$ such that $\DM_{\mrm{gm}}(k,R)^{\heartsuit_w} \simeq \mbf{Chow}_{\infty}(k,R)$.
\end{thm}
\begin{proof}
By \cite{perfection} we can assume $k$ is perfect.
By Remark~\ref{rem:constructing_weights} it suffices to check that the collection $\{ M(X)(q)[2q] \}$
where $X$ is a smooth projective $k$-scheme and $q \in \bbZ$ is a collection of generators for $\DM_{\mrm{gm}}(k,R)$ and that the mapping spectrum
$$\mrm{maps}_{\DM_{\mrm{gm}}(k; R)}(M(X)(n)[2n], M(Y)(m)[2m])$$
is connective for any smooth projective $X,Y$ and any $n,m \in \mathbb{Z}$.
The first fact follows from \cite[Proposition~5.3.3]{ShaneAsterisque}, the second fact was already proved in Lemma~\ref{lem:connective}.

\end{proof}

\begin{rem}\label{rem:relative-motives}
The analogous weight structure can be constructed more generally on the $\infty$-categories of relative geometric motives
$\DM_{\mrm{gm}}(S;\bbQ)$ (see \cite{bondarko-relative-rational}) for quasi-excellent
finite dimensional separated schemes $S$ and even on some subcategories of $\DM_{\mrm{gm}}(S;\bbZ[1/p])$, where $S$ is a scheme of exponential characteristic $p$ (see \cite{integral-cdh-motives}).
Moreover, similar weight structures can be constructed on the $\infty$-category of compact
$KGL_S$-modules and the $\infty$-category of compact $MGL_S$-modules (\cite[4]{weight-examples}).
\end{rem}

\sssec{Stacky examples}\label{ssec:stacks-weights}

In this section, we will furnish some examples from the context of algebraic stacks.

\begin{defn} \label{def:groups}
Let $k$ be a discrete commutative ring.
\begin{itemize}
\item
A {\bf group scheme} over $k$ is a group object in the category of $k$-schemes. A group scheme is said to be flat if it is flat as a $k$-scheme.

\item
A flat group scheme $G$ is called {\bf embeddable} if there exists a homomorphism
$G \to GL_n(k)$
which is a closed immersion \cite[Definition 2.1(1)]{ahr}.

\item
A flat group scheme $G$ is called {\bf linearly reductive} if the functor of taking invariants

$$\mbf{D}_{\mrm{qc}, BG} \to \mbf{D}_{\mrm{qc}, k}$$

is t-exact for the standard $t$-structure; this is the same as saying that the canonical morphism $\pi:BG \rightarrow \Spec k$ is cohomologically affine in the sense of \cite{alper}.

\item
A flat group scheme $G$ is called {\bf nice} if it is an extension
\[
1 \rightarrow G^0 \rightarrow G  \rightarrow H \rightarrow 1;
\]
where $G^0$ is closed, normal subgroup of multiplicative type (e.g. a torus) and $H$ is finite, locally constant with $|H|$ invertible in $k$ \cite[Definition 1.1]{HallRydh2}.
%
\end{itemize}
\end{defn}

We refer to \cite[Section 2]{ahr} for a more thorough exposition, but let us give some explanation and examples for the reader's convenience.

\begin{exam}\label{rem:nice} It is standard that a constant, finite group with order invertible on the base is linearly reductive. Furthermore, according to \cite[Proposition 12.17(2)]{alper}, linearly reductive groups are closed under extensions. Therefore any nice group scheme is linearly reductive.
\end{exam}

\begin{exam}\label{k:zero} Let $k$ be a ring of characteristic zero (so a $\bbQ$-algebra). Then, by \cite[Lemma 4.1.6, Remark 9.1.3]{alper2}, the notion of being linearly reductive is equivalent to the notion of being geometrically reductive (note that if $G$ is linearly reductive then it is always geometrically reductive). This latter notion is equivalent to $G$ being reductive in the classical sense whenever $G$ is smooth with geometrically connected fibers \cite[Theorem 9.7.5]{alper2}.
Hence, in this setting, the reader should feel free to substitute the notion of linear reductivity with the classical notion of reductivity.
\end{exam}

\begin{exam}\label{exam:k-p} Let $k$ be a field of characteristic $p > 0$; from the perspective of topological cyclic homology this is the primary situation of interest.
Then, being linearly reductive is equivalent to being nice \cite[Theorem 1.2]{HallRydh2}. Examples include tori, the roots of unity sheaves $\mu_{p^n}$ and constant group schemes of order prime to $p$. If $k$ is a ring of mixed characteristic, then \cite[Theorem 19.9]{ahr} gives a classification result: any linearly reductive group $G$ over $k$ is canonically an extension of a finite, tame linearly reductive group scheme by a smooth linearly reductive group scheme.
\end{exam}
%

If $X$ is a quasicompact and quasiseparated derived algebraic stack we recall $\Perf_X \subset \mbf{D}_{\mrm{qc},X}$ is the subcategory of spanned by those objects $M$ such that for any morphism from an affine derived scheme $x:\Spec A \rightarrow X$, $x^*M$ is a perfect $A$-module; such an object is a {\bf perfect complex} on $X$. The construction of weight structures on perfect complexes on stacks can be found in joint work of the second author with Bachmann, Khan and Ravi \cite[Example~5.1.2]{BKRS} which is a generalization of \cite[Theorem 3.14]{sosnilo2021regularity}. In all the cases of the next theorem, $\Perf_X$ coincides with compact objects in $\mbf{D}_{\mrm{qc},X}$ \cite[Theorems B and C]{HallRydh2}. Here we provide the details.

\begin{thm} \label{thm:stacks-weights}
Let $G$ be an embeddable linearly reductive group scheme over a commutative ring $k$.
\begin{enumerate}
\item
Let $R$ be a connective $\mathbb{E}_{\infty}$-$k$-algebra endowed with an action of $G$.
There is a bounded weight structure on $\Perf_{[\Spec R/G]}$ whose heart is the full subcategory
generated under retracts by the set
\[
N := \{ p^*\sE: \text{$\sE$ is a finite $G$-equivariant $k$-module \sE} \},
\]
where $p\colon [\Spec R/G] \rightarrow BG$ is the structure morphism.

\item
Let $R \stackrel{f}\to S$ be a $G$-equivariant morphism of connective $\mathbb{E}_{\infty}$-$k$-algebras endowed with an action of $G$.
Then the base change functor $\Perf_{[\Spec R/G]} \to \Perf_{[\Spec S/G]}$ is weight exact.
\end{enumerate}
\end{thm}
\begin{proof}
We apply the construction from Remark~\ref{rem:constructing_weights} to the set $N$. It suffices to show that
\begin{enumerate}
  \item[(a)] $\Perf_{[\Spec R/G]}$ is generated under finite limits, finite colimits and retracts by the
  set $N$,
  \item[(b)] for any $G$-equivariant $k$-modules $\sE$, $\sE'$ the spectrum
  $$\mrm{maps}_{\Perf_{[\Spec R/G]}}(p^*\sE, p^*\sE')$$
  is connective.
\end{enumerate}
Note first that the $\infty$-category $\mbf{D}_{\mrm{qc}, [\Spec R/G]}$ is
compactly generated by $N$. When $R=k$, this follows from \cite[Proposition~8.4]{HallRydh1} as $BG$ has
the resolution property.
This is also true in general since the morphism $[\Spec R/G] \stackrel{p}\to BG$ is affine
(see \cite[Remark~2.2.1]{BKRS}).
In particular, this implies that the subcategory of compact objects in $\mbf{D}_{\mrm{qc}, [\Spec R/G]}$ is
generated by $N$ under finite limits, finite colimits and retracts.
To prove (a) it suffices to show that every perfect complex is compact in
$\mbf{D}_{\mrm{qc}, [\Spec R/G]}$. The set $N$ contains the tensor unit of $\mbf{D}_{\mrm{qc}, [\Spec R/G]}$, in particular
the tensor unit is a compact object. Now the claim follows from \cite[Lemma~4.4(2)]{HallRydh1} whose proof works verbatim for derived stacks.
To prove (b) we use the equivalence
$$\mrm{maps}_{\Perf_{[\Spec R/G]}}(p^*\sE, p^*\sE') \cong \mrm{maps}_{\Perf_{R}}(p^*\sE, p^*\sE')^G$$
given by descent and the fact that taking invariants preserves connectivity by definition of linear reductivity.

The second part of the statement follows by construction. 
\end{proof}

\section{Nilpotent extensions of additive $\infty$-categories} \label{sec:nilp} In this section, we introduce nilpotent extensions of additive $\infty$-categories as a generalization of a nilpotent extension of connective $\bbE_1$-rings. We recall that the latter means an $\bbE_1$-morphism of connective $\bbE_1$-rings $R \rightarrow S$ such that the induced map $\pi_0(R) \rightarrow \pi_0(S)$ has nilpotent kernel. We will also provide an extended example in the form of square zero extensions of additive $\infty$-categories.

\ssec{The definition of a nilpotent extension of additive $\infty$-categories} We now introduce the first main definition of this paper.
\begin{defn} \label{def:nilp} A {\bf nilpotent extension} of additive $\infty$-categories is an additive functor between additive $\infty$-categories:
\[
f: \sA \rightarrow \sB,
\] such that:
\begin{enumerate}
\item $f$ is essentially surjective,
\item for all objects $x,y \in \sA$ the map
\[
\pi_0\Maps_{\sA}(x,y) \rightarrow \pi_0\Maps_{\sB}(f(x), f(y))
\]
is a surjection,
\item there exists $n \in \mathbb{N}$ such that
for any sequence of composable morphisms $f_1, \cdots, f_n \in \sA$ for which each $f(f_i)$ is trivial\footnote{We note that any additive $\infty$-category is pointed, hence a map $f: x \rightarrow y$ is said to be {\bf trivial} if it is homotopic to $f: x \rightarrow 0 \rightarrow y$.},
the composition $f_1\circ \cdots\circ f_n$ is trivial.
\end{enumerate}
The third condition is equivalent to saying that the kernel ideal of the functor $h(\sA) \to h(\sB)$ is nilpotent.
\end{defn}

\begin{rem}\label{rem:exponent} We note that the integer $n$ that appears in part (3) of Definition~\ref{def:nilp} only depends on the additive functor $f$.
\end{rem}


\begin{prop}\label{prop:cons} Suppose that $f: \sA \rightarrow \sB$ is a nilpotent extension of additive $\infty$-categories. Then $f$ is conservative.
\end{prop}

\begin{proof} Let $g: x \rightarrow y$ be a morphism in $\sA$. By hypothesis, $f(g)$ is invertible in $\sB$, hence we may pick an inverse $h': f(y) \rightarrow f(x)$. Using (2) of Definition~\ref{def:nilp}, we may choose a lift of $h'$ which we call $h$. We claim that $h$ is an inverse of $g$. Indeed, it suffices to prove that $gh$ and $hg$ are invertible. Since the arguments are the same we do it for $gh$. First note that $f(\id - gh) \simeq \id - f(g)f(h) \simeq 0$ since $f$ is an additive functor and $f(h)$ is an inverse of $f(g)$. Therefore, by assumption (3) of Definition~\ref{def:nilp}, there exists an $n$ such that $(\id - gh)^n \simeq 0$. But now, we can furnish and inverse to $gh$ given by
\[
(gh)^{-1} \simeq (\id - (\id - gh))^{-1} \simeq \id + (\id - gh) + (\id - gh)(\id - gh) + \cdots + (\id - gh)^{n-1}.
\]

\end{proof}

%

\begin{rem}\label{rem:bijection-on-objects}
After Proposition~\ref{prop:cons}, we can replace part (1) of Definition~\ref{def:nilp} by saying that the induced map on equivalence classes of objects:
\[
\pi_0(\sA^{\simeq}) \rightarrow \pi_0(\sB^{\simeq})
\]
is a bijection. 
\end{rem}

\begin{exam} \label{exam:one-object} Suppose that $R, S$ are connective $\mathbb{E}_1$-ring spectra. Then we can consider the additive $\infty$-categories $\mbf{Free}_R$, $\mbf{Free}_S$ of free modules over these rings. A morphism of $\mathbb{E}_1$-rings $f:R \rightarrow S$ then defines an additive functor $f: \mbf{Free}_R \rightarrow \mbf{Free}_S$. The requirement that the extension is nilpotent is then equivalent to asking that the map $\pi_0(R) \rightarrow \pi_0(S)$ has a kernel which is a nilpotent ideal.
\end{exam}

\begin{exam}\label{exam:hocat} What follows is arguably the most important example from the point of view of this paper in the same way that if $A$ is a ``derived ring" (an animated ring or a connective $\mathbb{E}_1$-algebra) then $A \rightarrow \pi_0(A)$ is a nilpotent extension. Let $\sA$ be an additive $\infty$-category and let $h(\sA)$ be its homotopy category which is a $1$-category. Then the map $\sA \rightarrow h(\sA)$ is evidently a nilpotent extension of additive $\infty$-categories.
\end{exam}


\ssec{Square zero extensions of additive $\infty$-categories}\label{sec:sqz} In this next section, we furnish a large collection of examples of nilpotent extensions of additive $\infty$-categories by introducing the notion of a square zero extension of additive $\infty$-categories. We emphasize that, unlike the proof of the DGM theorem for connective $\bbE_1$-ring spectra, this theory is \emph{not} necessary for the proof. On the other hand, they provide a large collection of new functors for which we can prove the DGM theorem.

\begin{defn} \label{def:bimodule} Let $\sA$ be an additive $\infty$-category. An {\bf $\sA$-bimodule} is a biadditive\footnote{By this we mean a functor which is additive in each variable. In the presence of tensor products of additive $\infty$-categories as in \cite[Sections 10.1.6, D.2.1.1]{SAG}, we can say that such a functor defines an additive functor $\sA^{\op} \otimes \sA \rightarrow \Spt_{\geq 0}$} functor
\[
\sM:\sA^{\op} \times \sA \rightarrow \Spt_{\geq 0}.
\]
By adjunction the functor gives rise to an additive functor
$$\sM:\sA \to \Fun^{\times}(\sA^{\op},\Spt)$$
which we refer to by the same letter.

\end{defn}

\begin{rem} \label{rem:bimodule}
Unpacking the structure of an $\sA$-bimodule, we get the following pieces of data:

\begin{enumerate}
\item Fixing an object $y \in \sA$, for any morphism $g: c \rightarrow c'$ in $\sA$ we get a map of connective spectra:
\[
g^*:\sM(c',y) \rightarrow \sM(c,y),
\]
as part of the functor
\[
\sM(-,y): \sA^{\op} \rightarrow \Spt_{\geq 0}.
\]
In particular, this endows $\sM(y,y) \in \Spt_{\geq 0}$ with the structure of a right $\mbf{End}_{\sA}(y,y)$-module.
\item Fixing an object $z \in \sA$, for any morphism $g: c \rightarrow c'$ in $\sA$ we get a map of connective spectra:
\[
g_*:\sM(z,c) \rightarrow \sM(z,c'),
\]
as part of the functor
\[
\sM(z,-): \sA \rightarrow \Spt_{\geq 0}.
\]
In particular, this endows $\sM(z,z) \in \Spt_{\geq 0}$ the structure of a left $\mbf{End}_{\sA}(z,z)$-module.
\item Altogether, we can package the two structures above by saying that for each  $z \in \sC$, $\sM(z,z)$ is an $\mbf{End}_{\sA}(z,z)$-bimodule.
\end{enumerate}
\end{rem}

We would like to make sense of the square zero extension of additive $\infty$-categories, which we denote as $\sA \oplus \sM$. To do so, we first recall the formalism of lax equalizers \cite[Section II.1]{nikolaus-scholze} which is a special case of lax pullbacks as in \cite[Section I]{tamme-excision}.

\begin{constr} \label{cons:land-tamme} Let $\sA, \sB$ be $\infty$-categories equipped with functors $F, G: \sA \rightarrow \sB$. The {\bf lax equalizer} of $F$ and $G$
\[
\mrm{LEq}(F, G)
\]
is the $\infty$-category given by the pullback
\[
\begin{tikzcd}\label{laxeq_constr}
\mrm{LEq}(F, G) \ar{d} \ar{r} & \Fun(\Delta^1,\sB) \ar{d}{(\on{res}_{\delta_0},\on{res}_{\delta_1})} \\
\sA \ar{r}{(F,G)} & \sB \times \sB.
\end{tikzcd}
\]
The objects of $\mrm{LEq}(F, G)$ are given by pairs $(c, f)$ where $c \in \sA$ and a morphism $f: F(c) \rightarrow G(c)$ in $\sB$. The mapping spaces in $\mrm{LEq}(F, G)$ are given by an equalizer diagram
\[
\Maps_{\mrm{LEq}(F, G)}((c,f), (c', f'))) \rightarrow \Maps_{\sA}(c, c') \rightrightarrows \Maps(F(c), G(c'))
\]
where one of the right maps is given by
\[
\Maps_{\sA}(c, c') \rightarrow \Maps_{\sB}(Fc, Fc') \xrightarrow{f'_*} \Maps_{\sB}(Fc, Gc'),
\]
and the other is
\[
\Maps_{\sA}(c, c') \rightarrow \Maps_{\sB}(Gc, Gc') \xrightarrow{f^*} \Maps_{\sB}(Fc, Gc').
\]
\end{constr}

\begin{lem}\label{lem:add-lax} Suppose that $\sA, \sB$ are additive $\infty$-categories with additive functors
\[
F, G:\sA \rightarrow \sB.
\] Then, $\mrm{LEq}(F, G)$ is an additive $\infty$-category and the functor $\mrm{LEq}(F, G) \rightarrow \sA$ is additive.
\end{lem}

\begin{proof}
$\Fun(\Delta^1,\sB)$ is additive since limits and colimits in a functor category are computed pointwise (see \cite[Corollary~5.1.2.3]{HTT}).
It suffices to show that a pullback of a diagram of additive $\infty$-categories is also an additive $\infty$-category.
This follows from general considerations about pullbacks of $\infty$-categories: if we have a cospan $\sA_0 \xrightarrow{f} \sA_1 \xleftarrow{g} \sA_2$, then $K$-shaped limits and colimits in the pullback $\infty$-category are computed pointwise if these limits and colimits are preserved by $f$ and $g$. This already proves that the pullback category is preadditive. To check the additivity condition we need only prove that the shear map is an equivalence, but equivalences in the pullback category are tested pointwise.
\end{proof}

\begin{constr} Given an additive $\infty$-category $\sA$ and an $\sA$-bimodule  $\sM$, we construct the {\bf square zero extension} $\sA \oplus \sM$ in the following manner: take the (pointwise) suspension of $\sM$ and adjoint to get a functor
\[
\sA \xrightarrow{\Sigma\sM} \Fun^{\times}(\sA^{\op}, \Spt).
\]
The additive $\infty$-category $\sA \oplus \sM$ is then defined as the lax equalizer between the Yoneda functor $\sA \stackrel{\yo}\rightarrow \Fun^{\times}(\sA^{\op}, \Spt)$ and $\Sigma\sM$:
\[
\sA \oplus \sM:=\mrm{LEq}(\yo, \Sigma\sM).
\]
In other words, it is the pullback
\[
\begin{tikzcd}
\sA \oplus \sM\arrow[r]\arrow[d,swap,"p"] & \Fun(\Delta^1,\Fun^{\times}(\sA^{\op}, \Spt))\arrow[d,"{(\on{res}_{\delta_0},\on{res}_{\delta_1})}"]\\
\sA\arrow[r, "{(\yo, \Sigma \sM)}"] &\Fun^{\times}(\sA^{\op}, \Spt) \times \Fun^{\times}(\sA^{\op}, \Spt).
\end{tikzcd}
\]

\end{constr}

\begin{lem}\label{lem:add} The $\infty$-category $\sA \oplus \sM$ is an additive $\infty$-category and the functor $p:\sA \oplus \sM \rightarrow \sA$ is additive.
\end{lem}

\begin{proof} This follows from Lemma~\ref{lem:add-lax}.
\end{proof}

\begin{rem}\label{rem:describe}
Let us unpack the $\infty$-category $\sA \oplus \sM$. Its objects are pairs
\[
x=(c \in \sA , t :\yo(c) \rightarrow \Sigma \sM(-,c)).
\]
Now, since $\sM$ takes values in connective spectra, we have that $\pi_0(\Sigma\sM(c,c)) = 0$, hence the data of $t$ is, in a sense made precise in Lemma~\ref{lem:redund}, redundant.
Moreover, we can compute the mapping (connective) spectrum
\[
\mbf{Maps}_{\sA \oplus \sM}((c,t),(c',t'))
\]
as the equalizer of the maps (notation as in Remark~\ref{rem:bimodule})
\begin{equation}\label{defining_maps}
\begin{tikzcd}
\mbf{Maps}_{\sA}(c,c') \arrow[rr, shift right, swap, "f\mapsto f^*t"]\arrow[rr,shift left, "f\mapsto f_*t'"] &&\Sigma \sM(c,c').
\end{tikzcd}
\end{equation}
This means that a morphism $x = (c, t) \rightarrow x' = (c', t')$ can be viewed as the data of a morphism $f:c \rightarrow c'$ in $\sA$ and a ``loop" in $\Sigma\sM(c, c')$ identifying $f^*t$ with $f_*t'$, i.e., a point in $\sM(c,c')$. 

The fiber sequence defining the equalizer also gives rise to a fiber sequence
\begin{equation}\label{extension_sequence}
\sM(p(x),p(x')) \stackrel{j_{x,x'}}\to \mbf{Maps}_{\sA \oplus \sM}(x,x') \stackrel{p_{x,x'}}\to \mbf{Maps}_{\sA}(p(x),p(x')).
\end{equation}
of functors landing in connective spectra
\[
(\sA\oplus \sM)^{\op} \times (\sA \oplus \sM) \rightarrow \Spt_{\geq 0}.
\]
\end{rem}

\begin{lem}\label{lem:section}
  The map $\sA \oplus \sM \rightarrow \sA$ admits a section
  \[
  i: \sA \rightarrow \sA \oplus \sM
  \]
  that sends an object $c$ to $(c,0)$.
\end{lem}
\begin{proof}
  The right vertical map in the diagram (\ref{laxeq_constr}) admits a section
  given by the functor $(b,b') \mapsto (b \stackrel{0}\to b')$.
  Formally, this is the direct sum of functors
  $$\sB \times \sB \stackrel{p_1}\to \sB = \Fun(\Delta^0, \sB) \stackrel{\on{Ran}_{\delta_0}}\to \Fun(\Delta^1, \sB) \text{ and }$$
  $$\sB \times \sB \stackrel{p_2}\to \sB = \Fun(\Delta^0, \sB) \stackrel{\on{Lan}_{\delta_1}}\to \Fun(\Delta^1, \sB),$$
  where $\on{Ran}$ (resp. $\on{Lan}$) is right Kan extension (resp. Left Kan extension).
  This induces a section of the left vertical map by functoriality of pullbacks. It sends
  $c$ to $(c,0)$ by construction.
\end{proof}

\begin{prop}\label{prop:sqzero_is_nil}
The functor $p$ defined above is a nilpotent extension of additive $\infty$-categories.
\end{prop}
\begin{proof}
Part (1) of Definition~\ref{def:nilp} follows from Lemma~\ref{lem:section}.
Observing the fiber sequence (\ref{extension_sequence}) and using the fact that $\pi_0\Sigma\sM(c,c') = 0$ we also see that Part (2) is satisfied.
Let us prove (3); we claim that the $n$ appearing in (3) is just $2$. Indeed, suppose that $f: x \rightarrow x', g: x' \rightarrow x''$ are composable morphisms in $\sA \oplus \sM$ such that $p(f) \simeq p(g) \simeq 0$. We claim that $g \circ f \simeq 0$. Indeed, by the bifunctoriality of the fiber sequence~\eqref{extension_sequence} we have the following commutative diagram where the rows are fiber sequences of spectra, all of which are connective:
\[
\begin{tikzcd}
\sM(p(x'),p(x'')) \ar[swap]{d}{f^*}  \ar{r}{j_{x',x''}} & \mbf{Maps}_{\sA \oplus \sM}(x',x'') \ar{d}{f^*} \ar{r} & \mbf{Maps}_{\sA}(p(x'), p(x'')) \ar{d}{f^*} \\
\sM(p(x), p(x'')) \ar{r}{j_{x,x''}} & \mbf{Maps}_{\sA \oplus \sM}(x, x'') \ar{r} & \mbf{Maps}_{\sA}(p(x), p(x'')).
\end{tikzcd}
\]
Now since $g \in \mbf{Maps}_{\sA \oplus \sM}(x',x'')$ is such that $p(g) \simeq 0$, we can find $m \in \sM(p(x'),p(x''))$ and an equivalence $j_{x',x''}(m) \simeq g$. On the other hand, since $p(f) \simeq 0$, we get that the left vertical $f^*$ is nullhomotopic. Therefore, $g \circ f \simeq f^*(g) \simeq j_{x,x''}\circ f^*(m) \simeq 0$.
\end{proof}

\begin{lem} \label{lem:redund} Consider the subcategory $\sA \oplus \sM^0 \subset \sA \oplus \sM$ spanned by objects of the form $(c \in \sA, 0)$. Then the inclusion functor $(\sA \oplus \sM)^0 \hookrightarrow  \sA \oplus \sM$ is an equivalence of categories.
\end{lem}
\begin{proof}
The functor in question is fully faithful, so it suffices to show essential surjectivity, i.e. that any object of $\sA \oplus \sM$ is equivalent to an object of the form $(c,0)$.
By Proposition~\ref{prop:sqzero_is_nil} $\sA \oplus \sM \stackrel{p}\to \sA$ is a nilpotent extension, in particular, it is surjective on homotopy classes of morphisms. Hence we can find a map $(c,t) \stackrel{f}\to (c,0)$ such that $p(f) \simeq \id_c$ for any object $(c,t) \in \sA \oplus \sM$.
By Proposition~\ref{prop:cons} $p$ is conservative, so $f$ is an equivalence.
\end{proof}


%

The next proposition gives us a way to compute composition in $\sA \oplus \sM$. We note that via the section $i: \sA \rightarrow \sA \oplus \sM$ in Lemma~\ref{lem:redund}, we get an additive functor
\[
\mbf{Maps}_{\sA \oplus \sM}(i(-),i(-)): \sA^{\op} \times \sA \rightarrow \Spt_{\geq 0}.
\]
Furthermore, by Lemma~\ref{lem:redund} again, any object in $\sA \oplus \sM$ is equivalent to one which is of the form $i(c)$.

\begin{prop}\label{prop:computing-sqzero-composite}
There is a natural equivalence of connective spectra, functorial in both variables:
$$\mbf{Maps}_{\sA \oplus \sM}(i(c),i(c')) \cong \mbf{Maps}_{\sA}(c,c') \oplus \sM(c,c').$$
Furthermore, the composition
$$i(c) \stackrel{(f_1, m_1)}\to i(c') \stackrel{(f_2, m_2)}\to i(c'')$$
in $\sA\oplus \sM$
can be computed as $(f_2f_1, f_{2*}m_1 + f^*_1m_2)$.
\end{prop}
\begin{proof}
We claim that the fiber sequence $i_*$\eqref{extension_sequence} of functors on $\sA^{\op} \times \sA$ splits. Indeed, restricting along $i$ amounts to setting $t$ and $t'$ in~\eqref{defining_maps} to be the maps classifiying the zero elements and thus the desired sequence splits from by its construction. This proves the first part of the statement.

To compute the composition, we note that since the composition operation is linear with respect to the additive $\mathbb{E}_{\infty}$-structure, we have that
\begin{eqnarray*}
(f_2,m_2) \circ (f_1,m_1) & \simeq & ((f_2, 0) + (0,m_2)) \circ (f_1, m_1)\\
& \simeq & (f_2,0) \circ (f_1, m_1) + (0, m_2)\circ(f_1, m_1)\\
& \simeq & (f_2,0) \circ (f_1, 0) + (f_2,0) \circ (0, m_1) + (0, m_2)\circ(f_1, 0) + (0,m_2)\circ(0, m_1).
\end{eqnarray*}
Hence the claim follows from the following sequence of equivalences:

\begin{enumerate}
\item $(f_2, 0) \circ (f_1, m) \simeq (f_2 \circ f_1, f_{2*}m)$,
\item $(f_2, m) \circ (f_1, 0) \simeq (f_2 \circ f_1, f^*_1m)$, and
\item $(0, m) \circ (0, m') \simeq 0$.
\end{enumerate}

We note that (3) was shown in the course of proving Proposition~\ref{prop:sqzero_is_nil}. To prove the first claim, because of the $\sA^{\op} \times \sA$-naturality of the decomposition, we get the commutative diagram of connective spectra:
\[
\begin{tikzcd}
\mbf{Maps}_{\sA \oplus \sM}(i(c'),i(c'')) \ar[swap]{d}{(f_1,0)^*} \ar{r}{\simeq} &  \mbf{Maps}_{\sA}(c',c'') \oplus \sM(c',c'')\ar{d}{f_1^* \oplus f_1^*}\\
\mbf{Maps}_{\sA \oplus \sM}(i(c),i(c'')) \ar{r}{\simeq} & \mbf{Maps}_{\sA}(c,c'') \oplus \sM(c,c'').
\end{tikzcd}
\]
This proves (1), while (2) follows from the commutativity of
\[
\begin{tikzcd}
\mbf{Maps}_{\sA \oplus \sM}(i(c),i(c')) \ar[swap]{d}{(f_2,0)_*} \ar{r}{\simeq} &  \mbf{Maps}_{\sA}(c,c') \oplus \sM(c,c')\ar{d}{f_{2*} \oplus f_{2*}}\\
\mbf{Maps}_{\sA \oplus \sM}(i(c),i(c'')) \ar{r}{\simeq} & \mbf{Maps}_{\sA}(c,c'') \oplus \sM(c,c'').
\end{tikzcd}
\]

\end{proof}

\begin{rem}\label{rem:comparison-with-dotto}
In \cite[Definition 1.3]{dotto}, Dotto defined the notion of square zero extension of $\mrm{Ab}$-enriched categories. We view our construction as an $\infty$-categorical analog of his construction. In contrast to his approach, our formulation does not begin by prescribing a composition law (which would be difficult to do in our setting). Rather, the composition law is a computation stemming from the presentation as a lax pullback. We can also view our definition as an $\infty$-categorical/spectral version of Tabuada's definition for dg categories \cite[Section 4]{tabuada-aq}.

Note that in the setting of additive categories the construction of Dotto is equivalent to
ours (by Lemma~\ref{lem:redund} and Proposition~\ref{prop:computing-sqzero-composite}).
\end{rem}

%

\begin{exam} Let $R$ be a connective $\bbE_1$-ring spectrum and $M$ a connective $R$-bimodule. Out of this, we can form a connective $\bbE_1$-$R$-algebra spectrum $R \oplus M$, the square-zero extension in the sense of \cite[Section 7.4.1]{HA}. The underlying $R$-module is given by $R \oplus M$. This generalizes the usual construction of square-zero extension when $R$ and $M$ are discrete; in this case $R \oplus M$ is endowed with multiplication given by:
$$(r_1,m_1), (r_2,m_2) \mapsto (r_1r_2, r_1m_2 + m_1r_2).$$

Now, $M$ determines a $\mbf{Proj}^\mrm{fg}_R$-bimodule $\sM$ 
$$\mbf{Proj}^{\mrm{fg},\op}_R \times \mbf{Proj}^{\mrm{fg}}_R \longrightarrow \Spt$$
$$(P_1, P_2) \mapsto \Maps(P_1,R) \otimes_R M \otimes_R P_2.$$
We claim that there is a canonical equivalence of additive $\infty$-categories
\[
\mbf{Proj}^\mrm{fg}_{R \oplus M} \simeq \mbf{Proj}^\mrm{fg}_R \oplus \sM.
\]
Indeed, by Lemma~\ref{lem:redund} $\Perf_R \oplus \sM$ is generated under finite sums and retracts by the object $(R,0)$.
Now, by Theorem~\ref{thm:Morita-theory} $\Perf_R \oplus \sM$ is equivalent to the $\infty$-category of projective modules over $\mbf{End}((R,0))$.
Moreover, the description of the endomorphism ring is given in Proposition~\ref{prop:computing-sqzero-composite}, whence we get that $\mbf{End}((R,0)) \simeq R \oplus M$ as connective $\bbE_1$-rings.
\end{exam}

\section{The DGM theorem for nilpotent extensions of additive $\infty$-categories} \label{sec:dgm-additive}

In this section, we prove a version of the DGM theorem for nilpotent extensions of additive $\infty$-categories.

\ssec{Localizing invariants of additive $\infty$-categories}
Following the convention of \cite{Land_2019} we say that a {\bf localizing invariant} is a functor
\[
E: \mbf{Cat}_{\infty}^{\perf} \rightarrow \Spt,
\]
which converts any exact sequence in $\mbf{Cat}_{\infty}^{\perf}$ to a cofiber sequence of spectra (exact sequences in the context of stable $\infty$-categories are recalled in Section~\ref{sect:comparestab}). This differs from the convention in \cite{blumberg2013universal} where a localizing invariant was also required to commute with filtered colimits. We say that $E$ is {\bf truncating} if for any connective $\bbE_1$-ring spectrum $R$, the canonical map
\[
R \rightarrow \pi_0(R)
\]
induces an equivalence
\[
E(R) \xrightarrow{\simeq} E(\pi_0(R)).
\]

Localizing invariants for additive $\infty$-categories are defined via the functor $(-)^{\mrm{fin}}$ and idempotent completion.

\begin{defn}\label{def:localizing-additive} Let $E$ be a localizing invariant and $\sA$ a small additive $\infty$-category. Then we set
\[
E^{\Sigma}(\sA) := E(\mrm{Kar}\left(\sA^{\fin}\right)).
\]

\end{defn}

\begin{rem}\label{rem:warning} Since a stable $\infty$-category $\sC$ is also an additive $\infty$-category, we can either take $E^{\Sigma}(\sC)$ or take $E(\sC)$. These spectra are, in general, \emph{not} equivalent. Furthermore, when $E = K$, the spectrum $K^{\Sigma}(\sA)$ is also not necessarily equivalent to the direct sum $K$-theory $K^{\oplus}(\sA)$; for example the latter is \emph{always} a connective spectrum while $K^{\Sigma}(\sA)$ need not be. This issue is further discussed in Appendix~\ref{app:kadd}. Using weight structures, we will be able to identify when $E(\sC) \simeq E^{\Sigma}(\sC^{\heartsuit_w})$ as we will explain in the next section. This is the relevance of the functor $E^{\Sigma}$.
\end{rem}

\begin{lem}\label{lem:cofibers} Suppose that
\[
\sB \rightarrow \sA \rightarrow \sC
\]
is an exact sequence of additive $\infty$-categories. Then for any localizing invariant $E$,
\[
E^{\Sigma}(\sB) \rightarrow E^{\Sigma}(\sA) \rightarrow E^{\Sigma}(\sC)
\]
is a cofiber sequence of spectra.
\end{lem}

\begin{proof} This follows from Lemma~\ref{lem:comparing_cofibers} and the definition of $E^{\Sigma}$.

\end{proof}

\ssec{Proof of DGM theorem for additive $\infty$-categories}
We now come to our main result in the setting of additive $\infty$-categories.

\begin{thm} \label{thm:main-add}  Let $f:\sA \rightarrow \sB$ be a nilpotent extension of small additive $\infty$-categories and let $E$ be a truncating invariant. Then, we have an induced equivalence of spectra
\[
E^{\Sigma}(\sA) \rightarrow E^{\Sigma}(\sB).
\]

%
\end{thm}


\begin{proof}
We will explain the following commutative diagram of \emph{small} additive $\infty$-categories

%
%
\begin{equation}\label{eq:main-diagram}
\begin{tikzcd}
\sA \ar{r} \ar{d} & \sA^{\mrm{big}} \ar{d} \ar{r} &  (\sA^{\mrm{big}}/\sA)^{\mrm{Kar}} \ar{d}\\
\sB \ar{r} & \sB^{\mrm{big}} \ar{r} & (\sB^{\mrm{big}}/\sB)^{\mrm{Kar}}.
\end{tikzcd}
\end{equation}
For a small additive $\infty$-category $\sE$, we consider
\[
\sE^{\mrm{big}} \subset \mrm{Ind}(\sE),
\]
which is the additive subcategory closed under retracts generated by
$$G_{\sE}=\bigoplus\limits_{x \in \mrm{Obj}(\sE)} x.$$ This is evidently a small additive $\infty$-category.

The canonical functor $\sE \rightarrow \mrm{Ind}(\sE)$ factors through $\sE^{\mrm{big}}$ since the latter category is closed under retracts, whence contains the Yoneda image. Note that this construction is functorial in small additive $\infty$-categories and additive functors. This explains the left square. The right square is just the induced map on cofibers, i.e., on Verdier quotients in the sense explained in Section~\ref{sec:add-verd}.

We note the following facts:
\begin{enumerate}
\item By Theorem~\ref{thm:Morita-theory}(2) we have an equivalence of additive $\infty$-categories
\[
\sA^{\mrm{big}} \simeq \mbf{Proj}^{\mrm{fg}}_{\mbf{End}_{\sA^{\mrm{big}}}(G_{\sA})},
\]
\item the additive $\infty$-category $\sA^{\mrm{big}}/\sA$ is generated by the image of $G_{\sA}$ under
finite sums and retracts. 
\end{enumerate}

Therefore applying Theorem~\ref{thm:Morita-theory}(2) again:
\[
(\sA^{\mrm{big}}/{\sA})^{\mrm{Kar}} \simeq \mbf{Proj}^{\mrm{fg}}_{\mbf{End}_{\sA^{\mrm{big}}/\sA}(G_{\sA})},
\]
Since $f$ is a nilpotent extension of additive $\infty$-categories (and is thus essentially surjective), we have that $G_{\sB}$ is a retract of $f(G_{\sA})$ and in particular
$f(G_{\sA})$ generates $\sB^{\mrm{big}}$ under finite sums and retracts.
Therefore, the analogous equivalences also hold for $f(G_{\sA})$ in $\sB^{\mrm{big}}$ and in $(\sB^{\mrm{big}}/\sB)^{\mrm{Kar}}$. Under these equivalences the middle and the right map of (\ref{eq:main-diagram}) correspond to base change functors along maps of $\mathbb{E}_1$-rings.

Now, the map
$$\pi_0\mbf{End}_{\sA^{\mrm{big}}}(G_{\sA}) \to \pi_0\mbf{End}_{\sB^{\mrm{big}}}(f(G_{\sA}))$$
is a nilpotent extension of rings. Indeed, $\pi_0\mbf{End}_{\sA^{\mrm{big}}}(G_{\sA})$ can be expressed as an infinite ring of matrices which are \emph{column finite} with entries in $\pi_0\Maps_\sA(x_i,x_j), x_i, x_j \in \sA$, and similarly for  $\pi_0\mbf{End}_{\sB^{\mrm{big}}}(f(G_{\sA}))$.
The kernel of this ring map consists of (infinite) matrices which are column finite and whose entries are elements $s \in \pi_0\Maps_\sA(x_i,x_j)$ such that $f(s) = 0$.
From the definition of a nilpotent extension of additive $\infty$-categories, this a nilpotent ideal. Moreover, the same is true for the kernel of the map
$$\pi_0\mbf{End}_{\sA^{\mrm{big}}/\sA}(G_{\sA}) \to \pi_0\mbf{End}_{\sB^{\mrm{big}}/\sB}(f(G_{\sA}))$$
by Corollary~\ref{cor:explicitquotcat}(2) which expresses each of the rings as a quotient of column-finite matrices modulo the ideal of matrices with only finitely many nonzero entries.

Now $E^{\Sigma}$ applied to the middle or the right vertical functor of~\eqref{eq:main-diagram} is an equivalence by the assumption that $E$ is truncating.
Since it is a localizing invariant it also induces an equivalence after applying it to $\sA \to \sB$ which
yields the claim.
\end{proof}

Here is a couple of corollaries. The next one is an $\infty$-categorical version of Dotto's theorem \cite{dotto}. This generalizes the split-exact case of his theorem by Remark~\ref{rem:comparison-with-dotto}.

\begin{cor}\label{cor:dotto} Suppose that $\sA$ is an additive $\infty$-category and $\sM$ is an $\sA$-bimodule. Then for any truncating invariant $E$ we have that
\[
E^{\Sigma}(\sA) \simeq E^{\Sigma}(\sA \oplus \sM).
\]
In particular if we set $E = K^{\mrm{inv},\Sigma}$, then we get a cartesian square:
\[
\begin{tikzcd}
K^{\Sigma}(\sA \oplus \sM) \ar{d} \ar{r} & TC^{\Sigma}(\sA \oplus \sM) \ar{d}\\
K^{\Sigma}(\sA) \ar{r} & TC^{\Sigma}(\sA).
\end{tikzcd}
\]
\end{cor}
The next corollary will be important for the sequel.

\begin{cor}\label{cor:add} Let $\sA$ be an addititive $\infty$-category. Then the functor $\sA \rightarrow h\sA$ induces an equivalence, for any truncating invariant $E$,
\[
E^{\Sigma}(\sA) \simeq E^{\Sigma}(h\sA).
\]
In particular if we set $E = K^{\mrm{inv},\Sigma}$, then we get a cartesian square:
\[
\begin{tikzcd}
K^{\Sigma}(\sA)\ar{d} \ar{r} & TC^{\Sigma}(\sA) \ar{d}\\
K^{\Sigma}(h(\sA)) \ar{r} & TC^{\Sigma}(h(\sA)).
\end{tikzcd}
\]
\end{cor}

%
%
%
%
\begin{rem} \label{rem:tc} $TC$ is not a finitary invariant. Indeed, as in \cite[Section 3.1]{lmt}, \cite[Section 2.4]{bcm}, we can write the subcategory of $p^{\infty}$-torsion objects in $\Perf_{\bbZ}$ as a colimit of categories of modules:
\[
\colim \Mod_{\bbZ/p^n\bbZ}(\Perf_{\bbZ}).
\]
However, $TC$ does not preserve this colimit as explained in \cite[Remark 3.27]{lmt}. Thus we cannot
prove DGM for a nilpotent extension of additive $\infty$-categories $\sA \rightarrow \sB$
simply by applying the original DGM theorem to the maps
$$\mbf{End}_{\sA}(x) \to \mbf{End}_{\sB}(x),$$
as $x$ varies through the objects of $\sA$ (and thus $\sB$) and then passing to the colimit. We note, however, the following amusing corollary of Theorem~\ref{thm:main-add}.
\end{rem}

\begin{cor}\label{cor:colim} The functor
\[
\mbf{Cat}^{\mrm{add}}_{\infty} \rightarrow \Spt \qquad \sA \mapsto \Fib(TC^{\Sigma}(\sA) \to TC^{\Sigma}(h\sA))
\] commutes
with filtered colimits.

\end{cor}

\begin{proof} After Theorem~\ref{thm:main-add}, this follows from the fact that the $K$-theory functor,
  the functor $\sA \mapsto h(\sA)$ and the functor $\sA \mapsto \mrm{Kar}(\sA^{\mrm{fin}})$ preserve filtered colimits.

\end{proof}

%
%
%
%

\section{The DGM theorem for weighted stable $\infty$-categories and applications} \label{sec:dgm-weighted}

So far we have only discussed additive $\infty$-categories and proved a DGM theorem in that context.
However, the examples we are mostly interested in come from the context of stable $\infty$-categories,
and their K-theory. Note that every stable $\infty$-category is also additive. The K-theory of a stable $\infty$-category is fundamentally different from its $K$-theory taken as an additive $\infty$-category in the sense of our previous section, i.e. we should not expect the two spectra $K^{\Sigma}(\sC)$ and $K(\sC)$ to be equivalent. Hence, our results do not directly apply.

Moreover, we do not quite know how to define nilpotent extensions of general stable $\infty$-categories. Note that if we regard an exact functor of stable $\infty$-categories as an additive functor of additive $\infty$-categories, then the following is a non-example of nilpotent extensions (in the sense of Definition~\ref{def:nilp}).
\begin{exam}The morphism of connective $\bbE_{\infty}$-rings $\mathbb{S} \rightarrow \bbZ$ defines, via base change, a functor $\Perf_{\mathbb{S}} \to \Perf_{\bbZ}$. This is not a nilpotent extension of
additive $\infty$-categories: the periodicity theorem in chromatic homotopy theory tells us that for any
non-trivial finite spectrum $X$ with torsion homology there are positive
degree maps $X \stackrel{f}\to \Sigma^d X$ that are trivial in homology such that $f^n$ is non-trivial for any $n>0$.
\end{exam}

On the other hand, we could try to relax the nilpotence condition by focusing on generators. This seems to lead to a wrong notion in the sense that one should not expect a DGM theorem for it, as explained in the next example.
\begin{exam} Let $k$ be any ring and consider the stable $\infty$-category $\Perf_{k^{\times 2}}$, which is generated by the two projective modules
$P_1, P_2$ given by the images of the two non-trivial idempotents in the ring $k^{\times 2}$.
The endomorphism ring spectrum of either of them is $k$.
Identifying $k^{\times 2}$ with diagonal matrices we get a ring homomorphism $k^{\times 2} \to M_2(k)$ which induces a
functor $\Perf_{k^{\times 2}} \to \Perf_{M_2(k)}$.
This functor is essentially surjective and induces an equivalence on endomorphism rings of each of
the generators.
However, the DGM theorem does not hold for the map $k^{\times 2} \to M_2(k)$.
\end{exam}

What we do instead is go through weight structures to define nilpotent extensions and show that the results of the previous section prove new cases of the DGM theorem in the context of stable $\infty$-categories.

\ssec{The DGM theorem for nilpotent extensions of weighted $\infty$-categories}\label{ssec:nilexamples}
\begin{defn}\label{def:nilp-stab} Let $(\sA, w), (\sB, w')$ be weighted $\infty$-categories. A weight exact functor $f: (\sA,w) \rightarrow (\sB,w')$ is said to be a {\bf nilpotent extension} if the functor of additive $\infty$-categories
\[
f:\sA^{\heartsuit_w} \rightarrow \sB^{\heartsuit_{w'}}
\]
is a nilpotent extension in the sense of Definition~\ref{def:nilp}.
\end{defn}

\begin{rem} \label{rem:idem} Theorem~\ref{thm:vova} in particular states that if $(\sA,w)$ is an idempotent complete stable $\infty$-category with a bounded weight structure, then the functor
\[
(\sA^{\heartsuit_w})^{\fin} \to \sA
\]
is an equivalence.
\end{rem}


\begin{thm}\label{thm:dgm-weights}
Let $f:(\sA,w) \rightarrow (\sB,w')$ be a nilpotent extension of boundedly weighted stable $\infty$-categories and let $E$ be a truncating invariant. Then we have an induced equivalence of spectra
\[
E(\sA) \rightarrow E(\sB).
\]

\end{thm}

\begin{proof} This follows immediately from Theorem~\ref{thm:main-add}, and Theorem~\ref{thm:vova} which gives identifications
\[
E(\sA) \simeq E^{\Sigma}(\sA^{\heartsuit_w}) \qquad E(\sB) \simeq E^{\Sigma}(\sB^{\heartsuit_{w'}}).
\]

\end{proof}

\begin{rem}\label{rem:k-linear-invariants}
We note that although Theorem~\ref{thm:dgm-weights} 
is proved for ``absolute" localizing invariants, its proof carries over to truncating invariants of $k$-linear boundedly weighted stable $\infty$-categories where $k$ is a connective $\bbE_{\infty}$-ring spectrum;
we refer to \cite[Remarks~1.18, 3.4]{Land_2019} for more discussions on this point.
\end{rem}

\begin{cor}\label{cor:dgm-weights} Let $f:(\sA,w) \rightarrow (\sB,w')$ be a nilpotent extension of boundedly weighted stable $\infty$-categories. Then we have an induced cartesian square of spectra:

\[
\begin{tikzcd}
K(\sA) \ar{d} \ar{r} & TC(\sA) \ar{d}\\
K(\sB) \ar{r} & TC(\sB).
\end{tikzcd}
\]
%
\end{cor}

\sssec{Specializations of results} We now discuss the examples that appear in the introduction.
\begin{exam}
For any boundedly weighted stable $\infty$-category $\sC$ one has a functor
$$\sC \simeq (\sC^{\heartsuit_w})^{\fin} \to (h\sC^{\heartsuit_w})^{\fin}$$
which we call the {\bf weight complex functor} (see \cite[Corollary~3.5]{vova-negative}).
This induces an equivalence on the homotopy categories of the heart, so it is a nilpotent extension of weighted categories.
The $\infty$-category $(h\sC^{\heartsuit_w})^{\fin}$ can also be described as the {\it homotopy category
of bounded complexes} in $h\sC^{\heartsuit_w}$, i.e. the localization of the 1-category of complexes by
the set of homotopy equivalences. Therefore, Theorem~\ref{thm:main-add} tells us that for any truncating invariant $E$, we have an induced equivalence
\[
E(\sC) \simeq E((h\sC^{\heartsuit_w})^{\fin}) = E^{\Sigma}(h\sC^{\heartsuit_w}).
\]
\end{exam}

\begin{exam}\label{exam:chow-weight-complex}
Since $h\mbf{Chow}_{\infty}$ is the classical additive category of Chow motives,
the previous example specializes to a functor
\[
\DM_{\mrm{gm}}(k;R) \rightarrow (\mbf{Chow}(k;R))^{\fin}.
\]
Applying Corollary~\ref{cor:dgm-weights} to this functor we obtain Corollary~\ref{cor:mot} from the introduction.
\end{exam}

\begin{exam}\label{exam:nilextension-of-stacks}
Now, let $k$ be a discrete ring and $G$ be an embeddable linearly reductive group scheme over $k$. Suppose that $R, S$ are connective $\bbE_{\infty}$-$k$-algebras endowed with $G$-actions and $f: R \rightarrow S$ is a $G$-equivariant map over $k$. Assume further that $\pi_0(R) \stackrel{\pi_0(f)}\to \pi_0(S)$ is a surjection with nilpotent kernel. Now Theorem~\ref{thm:stacks-weights}(1) endows both
$\Perf_{[\Spec R/G]}$ and $\Perf_{[\Spec S/G]}$ with weight structures and the induced functor $f^*: \Perf_{[\Spec R/G]} \rightarrow \Perf_{[\Spec S/G]}$ is weight exact by Theorem~\ref{thm:stacks-weights}(2). Below we prove that $f^*$ is a nilpotent extension of boundedly weighted stable $\infty$-categories. Applying Theorem~\ref{thm:dgm-weights} to this functor we obtain Corollary~\ref{cor:stacks} from the introduction.

\end{exam}
\begin{prop}\label{prop:nilp} With the notation of Example~\ref{exam:nilextension-of-stacks}, the functor
\[
f^*: \Perf_{[\Spec R/G]} \rightarrow \Perf_{[\Spec S/G]}
\]
is a nilpotent extension of boundedly weighted stable $\infty$-categories.
\end{prop}

\begin{proof}
Denote by
$$p_R \colon [\Spec R/G] \to BG,$$
$$p_S \colon [\Spec S/G] \to BG$$
the canonical projections and by $\bar{f}$ the map $[\Spec S/G] \to [\Spec R/G]$ induced by $f$.
By construction, the hearts of the weight structures on $\Perf_{[\Spec R/G]}$ and $\Perf_{[\Spec S/G]}$
are generated under finite
sums and retracts by pullbacks of finite $G$-equivariant $k$-modules, along $p_R$ and $p_S$, respectively.

By descent we may identify
$\Maps_{[\Spec R/G]}(p_R^*V,p_R^*W)$ with $\Maps_{R}(V,W)^G$ (and similarly for $S$).
Note that the map
$$\pi_0\Maps_R(V,W)^G \to \pi_0\Maps_S(f^*V,f^*W)^G$$
is surjective since $G$ is linearly reductive.

Let $n$ be an integer such that $I^n=0$ where $I=\Ker(\pi_0(R) \to \pi_0(S))$.
Up to adding direct summands the elements in the kernel of
$$\pi_0\Maps_R(V,W) \to \pi_0\Maps_S(f^*V,f^*W)$$
are defined by matrices with coefficients in $I$. Therefore, for any sequence of composable
morphisms  $f_1,\cdots,f_n$
in $\mbf{Proj}_R^{\mrm{fg}}$ such that $f^*(f_i)$ is trivial in $\mbf{Proj}_S^{\mrm{fg}}$, we get that $f_1 \circ \cdots \circ f_n = 0$.
Now since $\pi_0\Maps_R(V,W)^G$ is a subgroup in $\pi_0\Maps_R(V,W)$, the same is true for composable
sequences of morphisms in $\Perf_{[\Spec R/G]}^{\heartsuit_w}$ whose base change to $S$ is trivial.
So, the second and the third parts of Definition~\ref{def:nilp} are
satisfied.

Since $p_R \circ \bar{f} = p_S$, the restriction of $\bar{f}^*$ to the hearts
is essentially surjective up to retracts.
Any idempotent $p\in \pi_0\End_{[\Spec S/G]}(f^*V)$ can be lifted
along a nilpotent extension of rings
$$\pi_0\End_{[\Spec R/G]}(V) \to \pi_0\End_{[\Spec S/G]}(f^*V),$$
so $\bar{f}^*$ is essentially surjective and the first part of Definition~\ref{def:nilp} is also
satisfied.
\end{proof}

\begin{exam}\label{exam:derived} In the notation of Example~\ref{exam:nilextension-of-stacks}, let $R$ be a connective $\bbE_{\infty}$-$k$-algebra with a $G$-action. Then the map $[\Spec \pi_0(R)/G] \rightarrow [\Spec R/G]$ satisfies the hypotheses of Corollary~\ref{cor:stacks} and so our theorem applies to this situation. We will use this observation in Section~\ref{app:latt}.
\end{exam}

\ssec{Applications to truncating invariants of stacks} \label{sec:apps}
We now discuss some applications of Corollary~\ref{cor:stacks} to derived quasi-compact and quasi-separated algebraic stacks \footnote{Strictly following the convention of \cite{BKRS}, these are derived stacks built from animated rings. The reason for this choice was the usage of the theory of quasi-smooth morphisms in the proof of \mrm{cdh} descent results which is significantly simpler for animated rings. However, the relevant results from \cite{BKRS} which we need in this section work for spectral stacks as well, i.e., those built from connective  $\bbE_{\infty}$-rings as in \cite{SAG} and reader should feel free to work in that generality. Note that in Section~\ref{sec:classical}, all the stacks in sight are classical.}.
For this we use the Nisnevich topology on the
category of derived algebraic stacks, following the treatment in \cite{BKRS} where we refer the reader for details. At present, we give an outline of the relevant facts for the reader's convenience. By a {\bf Nisnevich square} of derived algebraic stacks \cite[Definition 2.1.1]{BKRS} we mean a cartesian diagram of derived algebraic stacks
\[
\begin{tikzcd}
U' \ar{d} \ar{r} & X' \ar{d}{f}\\
U \ar{r}{j} & X.
\end{tikzcd}
\]
where $j$ is a quasi-compact, quasi-separated open immersion and $f$ is a representable \'etale morphism of finite presentation inducing an equivalence $f^{-1}(Z) \rightarrow Z$ for some complementary (to $U$) closed substack $Z \hookrightarrow X$. The derived algebraic stacks that we care about are {\bf ANS}\footnote{This is an abbreviation of ``affine diagonal and nice stabilizers."}  in the sense of \cite[Appendix A.1]{BKRS}: these are derived algebraic stacks with an affine diagonal and nice (in the sense of Definition~\ref{def:groups}) stabilizers. Examples include global quotient stacks of separated derived schemes by a nice group scheme and separated Deligne-Mumford stacks; see \cite[Appendix A.1]{BKRS} for more examples and discussion. If $E$ is a localizing invariant we set
\[
E(X) = E(\Perf_X).
\]
Here are two key facts:
\begin{enumerate}
\item If $E$ is a localizing invariant, then it converts a Nisnevich square of ANS derived algebraic stacks to a cartesian square of spectra \cite[Corollary 4.1.2, Theorem A.3.2]{BKRS}; this is an extension of the fact that localizing invariants satisfy descent for derived algebraic spaces \cite[Appendix A]{cmmn}.
\item If $X$ is an ANS derived algebraic stack then there exists a \emph{finite} sequence of open immersions
\begin{equation}\label{eq:scallop}
\varnothing = U_0 \hookrightarrow U_1 \hookrightarrow \cdots  \hookrightarrow U_{n-1} \hookrightarrow U_n = X,
\end{equation}
an embeddable nice group scheme $G$ over an affine scheme $S$ and Nisnevich squares for $1 \leq i \leq n$:
\[
\begin{tikzcd}
W_i \ar{d} \ar{r} & V_i \ar{d}\\
U_{i-1} \ar{r} & U_i,
\end{tikzcd}
\]
where $V_i$ is affine over $BG$ and the map $V_i \rightarrow U_i$ is \'etale and affine. In particular $V_i$ is of the form $[\Spec R_i/G]$. This is precisely \cite[Theorem A.1.8]{BKRS} and is a derived version of the results of \cite{ahhr}.
\end{enumerate}
We will say that a closed immersion of derived algebraic stacks $X \to Y$ is a {\bf nilpotent extension} if the ideal sheaf is nilpotent. The next result is a global version of Corollary~\ref{cor:stacks}.


\begin{thm}\label{thm:global-nilexcision}
Let $Y$ be an ANS derived algebraic stack and let $X \to Y$ be a nilpotent extension of derived algebraic stacks.
Then the induced map $E(Y) \to E(X)$ is an equivalence for any truncating invariant $E$.
\end{thm}
\begin{proof} Let us first prove the result assuming that $Y$ is a global quotient stack $[\Spec R/G]$ where $G$ is a an embeddable, nice group scheme over an affine scheme $B$ and $\Spec R$ is a derived $B$-scheme. In this case, we can form the cartesian square
\[
\begin{tikzcd}
X' \ar{r} \ar{d} & \Spec R\ar{d}{p}\\
X \ar{r} & Y,
\end{tikzcd}
\]
where $\Spec R \rightarrow Y$ is the quotient map. Now, a closed immersion is an affine morphism and thus $X' \rightarrow \Spec R$ is an affine morphism. Therefore, we deduce that $X' \cong \Spec S$. On the other hand, $\Spec R \rightarrow Y$ is $G$-torsor and thus $X' \rightarrow X$ is a $G$-torsor as well and thus we conclude that $X = [\Spec S/G]$. The result, in this case, then follows from Corollary~\ref{cor:stacks}.

To prove the result in general, we consider a decomposition of $Y$ as in~\eqref{eq:scallop} and induct on $n$ which exists by point (2) above. More precisely, we say that an ANS derived algebraic stack $X$ is of \emph{length at most $n$} if it has a decomposition of the form~\eqref{eq:scallop} of length $n$. With this terminology, a global quotient stack is length at most zero and we have proved the result in this case in the previous paragraph (see Proposition~\ref{prop:nilp}). Let us assume that the result has been proved for all ANS derived algebraic stacks of length at most $n-1$ and let $Y$ be an ANS derived algebraic stack of length at most $n$. Choose a decomposition of the form~\eqref{eq:scallop} as above for $Y$. We have a morphism of cartesian squares induced by pullback along $X \rightarrow Y$.
\[
\begin{tikzcd}
W_{n} \ar{r} \ar{d} & V_n\ar{d}\\
U_{n-1} \ar{r} & Y
\end{tikzcd}
\Leftarrow
\begin{tikzcd}
W'_{n} \ar{r} \ar{d} & V'_n\ar{d}\\
U'_{n-1} \ar{r} & X
\end{tikzcd}
\]
Since all the morphisms above are representable and have affine diagonal,  \cite[Lemma A.1.7]{BKRS} tells us that all the derived stacks are ANS.
By construction, 
$U_{n-1}$ 
is length at most $n-1$. Therefore, by induction hypothesis, $E$ converts the map $U'_{n-1} \rightarrow U_{n-1}$ to an equivalence. Now $E$ also converts the map $V'_n \rightarrow V_n$ to an equivalence since $V_n$ is a global quotient stack. Lastly, $W_n$ is again an ANS derived algebraic stack of length at most $n-1$ (we can pullback a decomposition of $U_{n-1}$ to one on $W_{n}$) and thus the result also applies by induction hypothesis.  Since $E$ converts Nisnevich squares of ANS derived stacks to bicartesian square of spectra, the result is proved.

%
%
%
%
\end{proof}

Let $X$ be a derived algebraic stack, then we have its classical locus $X^{\mrm{cl}} \hookrightarrow X$; this morphism is a closed immersion and $X^{\mrm{cl}}$ is a classical algebraic stack. A global version of Example~\eqref{exam:derived} is as follows

\begin{cor}\label{cor:pizero} Let $X$ be an ANS derived algebraic stack. Then for any truncating invariant $E$, we have an induced equivalence
\[
E(X^{\mrm{cl}}) \simeq E(X).
\]
\end{cor}

\sssec{Cdh excision for truncating invariants}\label{sec:classical} In the next situation, we content ourselves with \emph{classical} algebraic stacks. Recall that an {\bf abstract blow-up square} of \emph{Noetherian} algebraic stacks
$$
\begin{tikzcd}\label{cdhsquarestacks}
Z \arrow[r]\arrow[d] & \tilde{X}\arrow[d, "p"] \\
Y        \arrow[r,"i"] &         X,
\end{tikzcd}
$$
is a commutative square with $p$ proper representable and $i$ closed immersion such that $p$ induces an isomorphism
$\tilde{X}- Z \to X-Y.$
Define $Y_n$ to be the $n$-th infinitesimal thickening of $Y$ in $X$ and $Z_n$ to be the $n$-th
infinitesimal thickening of $Z$ in $\tilde{X}$. One of the main results of \cite{BKRS} is the following pro-cdh descent statement:
\begin{thm}\label{cdhdescent}\cite{BKRS}
If $X$ is an ANS algebraic stack, then for any localizing invariant $\on{E}$ the square
$$
\begin{tikzcd}
\on{E}(X) \arrow[r]\arrow[d] & ``\lim" \on{E}(Y_n) \arrow[d] \\
\on{E}(\tilde{X}) \arrow[r] &         ``\lim" \on{E}(Z_n)
\end{tikzcd}
$$
is a weak pullback of pro-spectra.
\end{thm}

Now applying Theorem~\ref{thm:global-nilexcision} to the maps
$Z \to Z_n$ and $Y \to Y_n$ we see that the pro-systems in question trivialize.
In particular we obtain the following cdh descent result for ANS algebraic stacks:

\begin{cor} \label{cor:tr}
  In the notation of Theorem~\ref{cdhdescent}, the square
  $$
  \begin{tikzcd}
  \on{E}(X) \arrow[r]\arrow[d] & \on{E}(Y) \arrow[d] \\
  \on{E}(\tilde{X}) \arrow[r] &  \on{E}(Z)
  \end{tikzcd}
  $$
  is a pullback of spectra. Therefore truncating invariants of ANS algebraic stacks satisfy cdh descent.
\end{cor}

\begin{proof} The only point that needs explanation is the last part. We note that cdh descent on stacks is equivalent to a combination of Nisnevich excision and excision for abstract blowup squares; see \cite{BKRS} or \cite{Hoyois_2019} for details.

\end{proof}

\begin{exam} According to \cite[Proposition 3.14]{Land_2019}, the homotopy $K$-theory functor is a truncating invariant of $\mathbb{Z}$-linear stable $\infty$-categories. Setting $E = KH$ in Corollary~\ref{cor:tr} reproves the main result of Hoyois and Krishna \cite{Hoyois_2019}. Of course, Corollary~\ref{cor:tr} also proves cdh descent on stacks for other truncating invariants discussed in \cite{Land_2019} such as $K^{\mrm{inf}}_{\bbQ}$ \cite[Corollary 3.9]{Land_2019} and periodic cyclic homology $HP$ \cite[Corollary 3.11]{Land_2019} on stacks over characteristic zero rings. We note that we can apply our results in the linear setting
by Remark~\ref{rem:k-linear-invariants}.
\end{exam}

\ssec{Applications to the lattice conjecture}\label{app:latt} The next application is a contribution on the literature surrounding the lattice conjecture for topological $K$-theory in the sense of Blanc. We recall some terminology and refer the reader to \cite{blanc,konavalov} for details. We work in the context of $\mbf{Cat}^{\perf}_{\bbC} := \Mod_{\Perf_{\bbC}}(\mbf{Cat}^{\perf}_{\infty})$, the $\infty$-category of small stable idempotent complete $\bbC$-linear stable $\infty$-categories and an object of
$\mbf{Cat}^{\perf}_{\bbC}$ will hereon be referred to as a {\bf $\bbC$-linear category}. We refer to \cite[Remark 1.18]{Land_2019} for a discussion of localizing invariants in this context and \cite{HSS} for details. In any case, we use the following terminology from \cite[Remark 1.18]{Land_2019}: a functor $\on{E}: \mbf{Cat}^{\perf}_{\bbC} \rightarrow \Spt$ which takes exact sequences in $\mbf{Cat}^{\perf}_{\bbC}$ to cofiber sequences is a {\bf $\bbC$-linear localizing invariant}. It is furthermore {\bf truncating} if for all connective $\bbE_1$-$\bbC$-algebras (equivalently connective $\bbC$-dga's \cite[Proposition 25.1.2.2]{SAG}) $A$, the map $E(A) \rightarrow E(\pi_0(A))$ is an equivalence.

 The work of Blanc \cite{blanc} constructs the functor of {\bf topological $K$-theory}
\[
K^{\mrm{top}}: \mbf{Cat}^{\perf}_{\bbC} \rightarrow \Spt.
\]
This functor is a localizing invariant \cite[Theorem 1.1(c)]{blanc}. If $X$ is a separated $\bbC$-scheme of finite type then we have a canonical equivalence of spectra \cite[Theorem 1.1(b)]{blanc}
\[
K^{\mrm{top}}(\Perf_X) \simeq KU(X(\bbC)),
\]
where the right-hand-side is the complex topological $K$-theory spectrum of the analytic space associated to the $\bbC$-points of $X$. The lattice conjecture concerns a map of localizing invariants
from $K^{\mrm{top}}$ to periodic cyclic homology
called the {\bf Chern character} \cite[Theorem 1.1(d)]{blanc}
\[
\mrm{Ch}:K^{\mrm{top}} \rightarrow \mrm{HP}.
\]
The conjecture appears as \cite[Conjecture 1.7]{blanc}.

\begin{conj}\label{conj:lattice} Let $\sC$ be a smooth and proper $\bbC$-linear category. Then the map
\[
\mrm{Ch} \otimes H\bbC: K^{\mrm{top}}(\sC) \otimes H\bbC \rightarrow HP(\sC)
\]
is an equivalence.
\end{conj}

Conjecture~\ref{conj:lattice} belongs to the world of noncommutative Hodge theory and implies the existence of a suitable Hodge structure on the homotopy groups of $K^{\mrm{top}}$; see, for example, \cite[Proposition 5.4, Remark 5.5]{perry}. Recent work of Konovalov brought to bear trace methods into this problem \cite{konavalov}. His main insight is the following result.

\begin{thm}[Konovalov] \label{thm:andrei} Let $\sL$ be the fiber of the (complexified) Chern character
\[
\sL(-):= \mrm{Fib}(\mrm{Ch} \otimes H\bbC: K^{\mrm{top}}(-) \otimes H\bbC \rightarrow HP(-)).
\] Then $\sL$ is a $\bbC$-linear truncating invariant.
\end{thm}


Now combining our results with the results of Halpern-Leistner and Pomerleano \cite{dans-hodge}, which verify the lattice conjecture for some classical algebraic stacks, we get:

\begin{thm}\label{thm:lattice-for-stacks}
Let $X$ be a derived stack over $\bbC$ satisfying any of the following hypotheses:
\begin{enumerate}
\item
it is of the form $[Y/G]$ where $Y$ is a derived, affine $G$-scheme over $\bbC$, $G$ is a reductive $\bbC$-group scheme such that $Y^{\mrm{cl}}$ is a smooth $\bbC$-scheme; 
\item
it is ANS derived stack such that $X^{\mrm{cl}}$ is a smooth $\bbC$-stack.
\end{enumerate}
Then, the lattice conjecture holds for $\Perf_X$.

\end{thm}

\begin{proof} 
Applying Corollary~\ref{cor:stacks} and Example~\ref{k:zero} (for case (1)), and Theorem~\ref{thm:global-nilexcision} (for (2)), in conjunction with Theorem~\ref{thm:andrei},
we reduce to showing that $\sL(X) = 0$ for $X=X^{\mrm{cl}}$. Note that we can apply our results in the $\bbC$-linear setting by Remark~\ref{rem:k-linear-invariants}.

In case (1) is satisfied, the result follows from \cite[Theorem 3.20]{dans-hodge} whose hypotheses are verified under the hypotheses on $Y^{\mrm{cl}}$ by \cite[Theorem 2.3]{dans-hodge}.

In case (2) is satisfied, an inductive argument on length of the ANS stack $X$, as in the proof of Theorem~\ref{thm:global-nilexcision}, together with Nisnevich descent for $\sL$ reduces to the case (1).
\end{proof}

\appendix

\section{$K$-theory of additive $\infty$-categories} \label{app:kadd}

Let $\sA$ be an additive $\infty$-category. In this section, we clarify what it means to take $K(\sA)$ and prove that in the situations we are interested in, all notions of $K$-theory agree.  To $\sA$ we can attach the following spectra:

\begin{enumerate}
\item the {\bf direct sum} $K$-theory in the style of Segal \cite{segal1974categories} and revisited by Gepner-Groth-Nikolaus \cite[Section 8]{ggn}
\[
K^{\oplus}(\sA).
\]
It is obtained by first taking the core of $\sA$ to get an $\mathbb{E}_{\infty}$-monoid in spaces, $\sA^{\simeq}$, whose operation is induced by direct sum. Then $K^{\oplus}(\sA)$ is the \emph{connective} spectrum obtained by taking group completion and invoking the identification between group-complete $\mathbb{E}_{\infty}$-monoids in spaces with connective spectra.
\item As we have done for most of this paper we can apply the $K$-theory functor as in \cite{blumberg2013universal}, characterized as the universal localizing invariant, to $\sA^{\mrm{fin}}$; we denoted this by
\[
K^\Sigma(\sA).
\]
\item Suppose that $\sA \subset \sC$ is an additive subcategory of a stable $\infty$-category (for example, the weight-heart of a weight structure on $\sC$). Then we can equip $\sC$ with the structure of a {\bf Waldhausen $\infty$-category} in the sense of \cite{BarwickKtheory} by the {\bf maximal pair structure} of \cite[Example 2.11]{BarwickKtheory} insisting that all maps are ingressives. We can then induce the structure of a Waldhausen $\infty$-category on $\sA$ where the ingressives are those morphisms with cofibers in $\sA$. To this, we can attach the {\bf Waldhausen-Barwick $K$-theory}:
\[
K^{WB}(\sA),
\]
as constructed in \cite[Part 3]{BarwickKtheory}.
\end{enumerate}

\begin{thm} \label{thm:all-agree} Let $\sC$ be a boundedly weighted stable $\infty$-category and let $\sC^{\heartsuit_w} \subset \sC$ be its weight heart. Then:
\begin{enumerate}
\item There are canonical equivalences of connective spectra
\[
K^{\oplus}(\sC^{\heartsuit_w}) \simeq K^{WB}(\sC^{\heartsuit_w}) \simeq K^{WB}(\sC)
\]
\item There is a canonical morphism
\[
K^{WB}(\sC^{\heartsuit_w}) \to K^{\Sigma}((\sC^{\heartsuit_w})^{\mrm{fin}})
\]
which identifies as a connective cover.
\end{enumerate}
\end{thm}

\begin{proof} Using \cite[Theorem 5.1]{fontes2019weight} or \cite[Theorem A.15]{aron}, we have an equivalence of connective spectra
\[
K^{WB}(\sC) \simeq K^{WB}(\sC^{\heartsuit_w}).
\]
Using \cite[Proposition A.19]{aron}, noting that ingressives are split in $\sC^{\heartsuit_w}$, we have a further identification:
\[
K^{\oplus}(\sC^{\heartsuit_w}) \simeq K^{WB}(\sC^{\heartsuit_w}).
\]

For the second statement we note that, by Theorem~\ref{thm:vova}, we have an equivalence of nonconnective spectra:
\[
K(\sC) \simeq K((\sC^{\heartsuit_w})^{\mrm{fin}}).
\]
On the other hand, the connective version of $K$-theory of the stable $\infty$-category $\sC$ in the sense of \cite{blumberg2013universal} is, by construction, the same as $K^{WB}(\sC)$ with the maximal pair structure. Therefore, we have a natural map to the nonconnective $K$-theory of \cite{blumberg2013universal}
\[
K^{WB}(\sC) \rightarrow K(\sC)
\]
which witnesses a connective cover. Therefore, the map induced by the equivalence of the first part:
\[
K^{WB}(\sC^{\heartsuit_w}) \simeq K^{WB}(\sC) \rightarrow K((\sC^{\heartsuit_w})^{\mrm{fin}}),
\]
is indeed a connective cover.
\end{proof}


%
%
%
%
%
%
%
\let\mathbb=\mathbf

{\small
\bibliography{references}
}

\parskip 0pt

\end{document}